\documentclass[11pt]{amsart}
\usepackage{amsmath,amssymb,xypic,array}
\usepackage[T1]{fontenc}
\usepackage{amsfonts}
\usepackage{amsmath}
\usepackage{amssymb}
\usepackage{amsthm}
\usepackage{setspace}
\usepackage[cp850]{inputenc}
\usepackage{hyperref}
\usepackage{graphicx}
\usepackage{fancyhdr}

\usepackage{tikz}
\usepackage{booktabs}
\usepackage{longtable}
\usepackage{geometry}
\usetikzlibrary{trees}

\providecommand{\U}[1]{\protect\rule{.1in}{.1in}}
\providecommand{\U}[1]{\protect\rule{.1in}{.1in}}
\providecommand{\U}[1]{\protect\rule{.1in}{.1in}}
\providecommand{\U}[1]{\protect\rule{.1in}{.1in}}
\providecommand{\U}[1]{\protect\rule{.1in}{.1in}}
\input{xy}
\xyoption{all}
\theoremstyle{theorem}

\newtheorem{Theorem}{Theorem}[section]

\newtheorem{Lemma}[Theorem]{Lemma}
\newtheorem{Proposition}[Theorem]{Proposition}
\newtheorem{Corollary}[Theorem]{Corollary}
\newtheorem{Conjecture}[Theorem]{Conjecture}

\theoremstyle{definition}

\newtheorem{Definition}[Theorem]{Definition}
\newtheorem{Remark}[Theorem]{Remark}

\newtheorem{Example}[Theorem]{Example}

\newtheorem{Claim}[Theorem]{Claim}

\numberwithin{equation}{section}

\newcommand{\np}{\noindent}

\def\leq{\leqslant}
\def\geq{\geqslant}

\def\bibaut#1{{\sc #1}}

\def\phi{\varphi}
\def\ro[#1]{{\textcolor{red}{#1}}}

\begin{document}

\title{A class of adding machine and Julia sets}

\begin{abstract}
In this work we define a stochastic adding machine associated to the
Fibonacci base and to a probabilities sequence
$\overline{p}=(p_i)_{i\geq 1}$. We obtain a Markov chain whose
states are the set of nonnegative integers. We study probabilistic
properties of this chain, such as transience and recurrence. We also
prove that the spectrum associated to this Markov chain is connected
to the fibered Julia sets for a class of endomorphisms in
$\mathbb{C}^2$.
\end{abstract}

\author[Danilo Antonio Caprio]{Danilo Antonio Caprio}
\address{\sc Danilo Antonio Caprio\\
Instituto de Matem\'atica Pura e Aplicada\\
Estrada Dona Castorina 110\\
22460-320\\ Rio de Janeiro, RJ, Brasil} \email{caprio@impa.br}
\date{\today}

\maketitle
\tableofcontents

\section*{Introduction}

The goal of this work is to study a problem in the areas of dynamical
systems, probability and spectral operator theory. These problems
involve stochastic adding machine, spectrum of transition operator
associated to Markov chains and fibered Julia sets of endomorphisms
in $\mathbb{C}^2$.

Let $g:\mathbb{C}\longrightarrow \mathbb{C}$ be an holomorphic map.
The filled Julia set associated to $g$ is by definition the set
$K(g)=\{z\in\mathbb{C}: g^n(z) \textrm{ is bounded}\}$, where $g^n$
is the $n-$th iterate of $g$.

A connection between Julia sets in $\mathbb{C}$ and stochastic
adding machine was given by Killeen and Taylor in \cite{kt1}. They
defined the stochastic adding machine as follows: let
$N\in\mathbb{N}=\{0,1,2,\ldots\}$ be a nonnegative integer. By using
the greedy algorithm we may write $N$ as
$N=\sum_{i=0}^{k(N)}\varepsilon_i(N)2^i$, in a unique
way, where $\varepsilon_i(N)\in\{0,1\}$, for all
$i\in\{0,\ldots,k(N)\}$. So, the representation of $N$ in base $2$
is given by $N=\varepsilon_{k(N)}(N)\ldots\varepsilon_0(N)$. They
defined a systems of evolving equation that calculates the digits of
$N+1$ in base $2$, introducing an auxiliary variable "carry $1$",
$c_i(N)$, for each digit $\varepsilon_i(N)$, as follows:

Define $c_{-1}(N+1)=1$ and for all $i\geq 0$, do

\begin{eqnarray}
\varepsilon_i(N+1)=(\varepsilon_{i}(N)+c_{i-1}(N+1))\mod 2;  \label{desc}\\
c_i(N+1)=\left[\frac{\varepsilon_i(N)+c_{i-1}(N+1)}{2}\right],
\nonumber
\end{eqnarray}
\np where $[x]$ is the integer part of $x\in\mathbb{R}^+$.

Killeen and Taylor \cite{kt1} defined a stochastic adding machine
considering a family of independent, identically distributed randon
variables $\{e_i(n):i\geq 0, \textrm{ } n\in\mathbb{N}\}$,
parametrized by nonnegative integers $i$ and $n$, where each
$e_i(n)$ takes the value $0$ with probability $1-p$ and the value $1$
with probability $p$. More precisely, they defined a stochastic
adding machine as follows: let $N$ be a nonnegative integer and
consider the sequences $(\varepsilon(N+1))_{i\geq 0}$ and
$(c_i(N+1))_{i\geq -1}$ defined by $c_{-1}(N+1)=1$ and for all
$i\geq 0$
\begin{eqnarray}
\varepsilon_i(N+1)=(\varepsilon_{i}(N)+e_i(N)c_{i-1}(N+1))\mod 2; \label{descp} \\
c_i(N+1)=\left[\frac{\varepsilon_i(N)+e_i(N)c_{i-1}(N+1)}{2}\right].
\nonumber
\end{eqnarray}
Killeen and Taylor \cite{kt1} studied the spectrum of the transition
operator $S$ associated to the stochastic adding machine in base
$2$, acting in $l^\infty(\mathbb{N})$, and they proved that the spectrum of $S$ in $l^\infty(\mathbb{N})$ is
equal to the filled Julia set of the quadratic map
$f:\mathbb{C}\longrightarrow\mathbb{C}$ defined by
$f(z)=\left(\frac{z-(1-p)}{p}\right)^2$, i.e.
$\sigma(S)=\{z\in\mathbb{C}:(f^n(z))_{n\geq 0} \textrm{ is
bounded}\}$.

In \cite{msv} and \cite{mv}, the authors considered the stochastic
adding machine taking a probabilities sequence
$\overline{p}=(p_i)_{i\geq 1}$, where the probability change in each
state, i.e. on the description (\ref{descp}) we have $e_i(N)=1$ with
probability $p_{i+1}$ and $e_i(N)=0$ with probability $1-p_{i+1}$,
for all $i\geq 0$, and they constructed the transition operator
$S_{\overline{p}}$ related to this probabilities sequence.

In particular, in \cite{msv} they proved that the Markov chain is null recurrent if only if
$\prod_{i=1}^{+\infty}p_i=0$. Otherwise the chain is transient. Furthermore, they proved that the spectrum of $S_{\overline{p}}$ acting in $c_0$,
$c$ and $l^\alpha$, $1\leq\alpha\leq\infty$, is equal to the fibered
Julia set
$E_{\overline{p}}:=\left\{z\in\mathbb{C}: (\tilde{f}_j(z)) \textrm{ is bounded}\right\}$,
where $\tilde{f}_j:=f_j\circ\ldots\circ f_1$ and
$f_j:\mathbb{C}\longrightarrow\mathbb{C}$ are maps defined by
$f_j(z)=\left(\frac{z-(1-p_j)}{p_j}\right)^2$, for all $j\geq 1$. Moreover, the spectrum of $S_{\overline{p}}$ in $l^\infty(\mathbb{N})$ is equal to the point spectrum.

In this paper, instead of base $2$, we will consider the Fibonacci
base $(F_n)_{n\geq 0}$ defined by $F_n=F_{n-1}+F_{n-2}$, for all
$n\geq 2$, where $F_0=1$ and $F_1=2$. Also, we will consider a
probabilities sequence $\overline{p}=(p_i)_{i\geq 1}$, instead of an
unique probability $p$ (as was done in \cite{ms} and \cite{um}, for a large class
of recurrent sequences of order $2$). Thenceforth, we will define
the Fibonacci stochastic adding machine and considering the transition operator $S$, we will prove that the Markov chain is transient if only if $\prod_{i=1}^{\infty}p_i>0$. Otherwise, if $\sum_{i=1}^{+\infty}p_i=+\infty$ then the Markov chains is null recurrent and if $\sum_{i=2}^{+\infty}p_iF_{2(i-1)}<+\infty$ then
the Markov chain is positive recurrent.

We will compute the point spectrum and prove that it is connected to the fibered Julia sets for a class of endomorphisms in $\mathbb{C}^2$. Precisely, $\sigma_{pt}(S)\subset E \subset \sigma_a(S)$ where $E=\{z\in\mathbb{C}:(g_n\circ\ldots\circ g_0(z,z))_{n\geq 0}\textrm{ is bounded}\}$ and $g_n:\mathbb{C}^2\longrightarrow\mathbb{C}^2$ are maps defined by $g_0(x,y)=\left(\frac{x-(1-p_1)}{p_1},\frac{y-(1-p_1)}{p_1}\right)$ and $g_n(x,y)=\left(\frac{1}{r_n}xy-\left(\frac{1}{r_n}-1\right),x\right)$ for all $n\geq 1$, where $r_n=p_{\left[\frac{n+1}{2}\right]+1}$. Moreover, if $\liminf_{i\to+\infty}p_i>0$ then $E$ is compact and $\mathbb{C}\setminus E$ is connected.

The paper is organized as follows. In subsection 1.1 we give the definition of Fibonacci adding machine and explain how this machine is defined by a finite transductor. In subsection 1.2 we define the stochastic Fibonacci adding machine and we obtain the transition operator of the Markov chain. In section 2 we give a necessary and sufficient condition for transience, a sufficient condition for null recurrence and for positive recurrence. In section 3 is devoted to provide an exact description of the spectra of these transition operators acting on $l^\infty(\mathbb{N})$. The section 4 contains results about connectedness properties of the fibered Julia sets.

\section{Stochastic Fibonacci adding machine}\label{fam}

\subsection{Adding machine}\label{amt}

Let $(F_i)_{i\geq 0}$ be the Fibonacci sequence defined by $F_0=1$,
$F_1=2$ and $F_n=F_{n-1}+F_{n-2}$, for all $n\geq 2$. By using the
greedy algorithm we can write every nonnegative integer number $N$,
in a unique way, as $N=\sum_{i=0}^{k(N)}\varepsilon_i(N)F_i$, where
$\varepsilon_i(N)=0$ or $1$ and
$\varepsilon_{i+1}(N)\varepsilon_i(N) \neq 11$, for all $0\leq i<
k(N)$.

\begin{Example}
$a)$ $12=8+3+1=F_4+F_2+F_0=10101$; $b)$ $14=13+1=F_5+F_1=100010$.
\end{Example}

A way to calculate the digits of $N+1$ in Fibonacci base is
given by a finite transducer $\mathcal{T}$ on $A^*\times A^*$, where
$A=\{0,1\}$ is a finite alphabet and $A^*$ is the set of finite
words on $A$. The transducer $\mathcal{T}$ of the Fibonacci adding
machine, represented on figure \ref{maquinafib}, is formed by two
states, an initial state $I$ and a terminal state $T$. The initial
state is connected to itself by two arrows. One of them is labelled
by $(10/00)$ and the other by $(1/0)$. There is also one arrow going
from the initial state to the terminal one. This arrow is labelled
by $(00/01)$. The terminal state is connected to itself by two
arrows. One of them is labelled by $(0/0)$ and the other by $(1/1)$.

\begin{Remark}
The transducer $\mathcal{T}$ was defined in \cite{um}.
\end{Remark}

\begin{figure}[!h]
\centering
\includegraphics[scale=0.4]{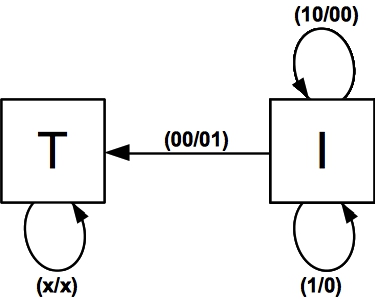}
\caption{Transductor of adding machine in Fibonacci base.}
\label{maquinafib}
\end{figure}
\begin{Example}
If $N=17=100101$ then $N$ corresponds to the path 
\begin{center}
$(T,1/1,T)(T,00/01,I)(I,10/00,I)(I,1/0,I)$.
\end{center}
Hence, $N+1=1010000=18$.
\end{Example}

Like was defined by Killeen and Taylor in description (\ref{desc}), we may construct an algorithm that computes the digits of $N+1$.
This algorithm can be described by introducing an auxiliary binary
"carry" variable $c_i(N+1)$ for each digit $\varepsilon_i(N)$ by the
following manner:

Put $c_{-1}(N+1)=1$ and

$i)$ If $\varepsilon_0(N)=0$ then for all $i\geq 0$, do:
\begin{eqnarray}
\varepsilon_{2i}(N+1)=\left[\frac{\varepsilon_{2i}(N)+c_{i-1}(N+1)}{\varepsilon_{2i+1}(N)c_{i-1}(N+1)+1}\right]; \nonumber \\
\varepsilon_{2i+1}(N+1)=\left[\frac{\varepsilon_{2i+1}(N)}{c_{i-1}(N+1)+1}\right]; \label{rev1}\\
c_i(N+1)=c_{i-1}(N+1)\varepsilon_{2i+1}(N). \nonumber
\end{eqnarray}
$ii)$ If $\varepsilon_0(N)=1$ then put $\varepsilon_0(N+1)=0$,
$c_0(N+1)=1$ and, for all $i\geq 1$, do:
\begin{eqnarray}
\varepsilon_{2i-1}(N+1)=\left[\frac{\varepsilon_{2i-1}(N)+c_{i-1}(N+1)}{\varepsilon_{2i}(N)c_{i-1}(N+1)+1}\right]; \nonumber \\
\varepsilon_{2i}(N+1)=\left[\frac{\varepsilon_{2i}(N)}{c_{i-1}(N+1)+1}\right]; \label{rev2}\\
c_i(N+1)=c_{i-1}(N+1)\varepsilon_{2i}(N).\nonumber
\end{eqnarray}

\begin{Remark}
If $c_{j-1}(N+1)=0$ for some nonnegative integer $j\geq 1$ then $c_i(N+1)=0$ for all $i\geq j-1$. Moreover, in description (\ref{rev1}) we have $\varepsilon_l(N+1)=\varepsilon_l(N)$ for all $l\geq 2j$ and in description (\ref{rev2}) we have $\varepsilon_l(N+1)=\varepsilon_l(N)$ for all $k\geq 2j-1$. \label{obsnao}
\end{Remark}

\begin{Example}
$a)$ Let $N=3=0100=\varepsilon_3(3)\varepsilon_2(3)\varepsilon_1(3)\varepsilon_0(3)$. By description (\ref{rev1}) we have $c_{-1}(4)=1$ and
$$
\left\{\begin{array}{l}
\varepsilon_0(4)=\left[\frac{\varepsilon_0(3)+c_{-1}(4)}{\varepsilon_1(3)c_{-1}(4)+1}\right] =1; \\
\varepsilon_1(4)\left[\frac{\varepsilon_1(3)}{c_{-1}(4)+1}\right]=0; \\
c_0(4)=c_{-1}(4)\varepsilon_1(3)=0;
\end{array}\right.  \textrm{ and } \left\{\begin{array}{l}
\varepsilon_2(4)=\left[\frac{\varepsilon_2(3)+c_{0}(4)}{\varepsilon_3(3)c_{0}(4)+1}\right] =1; \\
\varepsilon_3(4)\left[\frac{\varepsilon_3(3)}{c_{0}(4)+1}\right]=0; \\
c_1(4)=c_{0}(4)\varepsilon_3(3)=0.
\end{array}\right.
$$

Hence, $\varepsilon_3(4)\varepsilon_2(4)\varepsilon_1(4)\varepsilon_0(4)=0101=4=N+1$.

\np $b)$ Let $N=4=00101=\varepsilon_4(4)\varepsilon_3(4)\varepsilon_2(4)\varepsilon_1(4)\varepsilon_0(4)$. By description (\ref{rev2}) we have $\varepsilon_0(5)=0$, $c_{0}(5)=1$ and
$$
\left\{\begin{array}{l}
\varepsilon_1(5)=\left[\frac{\varepsilon_1(4)+c_{0}(5)}{\varepsilon_2(4)c_{0}(5)+1}\right] =0; \\
\varepsilon_2(5)\left[\frac{\varepsilon_2(4)}{c_{0}(5)+1}\right]=0; \\
c_1(5)=c_{0}(5)\varepsilon_2(4)=1;
\end{array}\right.  \textrm{ and } \left\{\begin{array}{l}
\varepsilon_3(5)=\left[\frac{\varepsilon_3(4)+c_{1}(5)}{\varepsilon_4(4)c_{1}(5)+1}\right] =1; \\
\varepsilon_4(5)\left[\frac{\varepsilon_4(4)}{c_{1}(5)+1}\right]=0; \\
c_2(5)=c_{1}(5)\varepsilon_4(5)=0.
\end{array}\right.
$$

Hence, $\varepsilon_4(5)\varepsilon_3(5)\varepsilon_2(5)\varepsilon_1(5)\varepsilon_0(5)=01000=5=N+1$.
\end{Example}

\begin{Theorem}
Let $N$ be a nonnegative integer number and $\varepsilon_k(N)\ldots
\varepsilon_0(N)$ its representation in Fibonacci base. The
algorithms (\ref{rev1}) and (\ref{rev2}) give for us the digits of
$N+1$ in Fibonacci base.
\end{Theorem}
\begin{proof} Let $N=\varepsilon_k(N)\ldots \varepsilon_0(N)$ and assume that $\varepsilon_0(N)=0$.

If $\varepsilon_1(N)=0$ then by de system (\ref{rev1}) we have $\varepsilon_0(N+1)=1$,
$\varepsilon_1(N+1)=0$, $c_0(N+1)=0$ and, by the remark \ref{obsnao},
$\varepsilon_i(N+1)=\varepsilon_i(N)$, for all $i\geq 2$. Therefore,
$\varepsilon_k(N+1)\ldots\varepsilon_2(N+1)
\varepsilon_1(N+1)\varepsilon_0(N+1)=
\varepsilon_k(N)\ldots\varepsilon_2(N)01=N+1$.

If $N=\varepsilon_k(N)\ldots\varepsilon_{2l+2}(N)00(10)^l$, for some
$l\geq 1$ then $\varepsilon_{2i+1}(N)\varepsilon_{2i}(N)=10$, for all
$i\in\{0,\ldots, l-1\}$, and
$\varepsilon_{2l+1}(N)\varepsilon_{2l}(N)=00$. Thus, by the system
(\ref{rev1}), we have $\varepsilon_{2i+1}(N)\varepsilon_{2i}(N)=00$
and $c_i(N+1)=1$, for all $i\in\{0,\ldots, l-1\}$,
$\varepsilon_{2l+1}(N)\varepsilon_{2l}(N)=01$ and $c_l(N+1)=0$.
Therefore, by remark \ref{obsnao},
$\varepsilon_i(N+1)=\varepsilon_i(N)$, for all $i\geq 2l+2$, and
$\varepsilon_k(N+1)\ldots\varepsilon_{2l+2}(N+1)
\varepsilon_{2l+1}(N+1)\varepsilon_{2l}(N+1)\varepsilon_{2l-1}(N+1)
\ldots\varepsilon_0(N+1)= \\
\varepsilon_k(N)\ldots\varepsilon_{2l+2}(N)01\underbrace{0\ldots
0}_{2l-1}=N+1$.

The case $\varepsilon_0(N)=1$ can be done by the same way using the system (\ref{rev2}).
\end{proof}

\subsection{Stochastic adding machine}\label{sam}

A way to construct the stochastic Fibonacci adding machine is
by a "probabilistic" transducer $\mathcal{T}_{\overline{p}}$ (see
Fig.\ref{maqpi}).

The states of $\mathcal{T}_{\overline{p}}$ are $T$ and $I_i$, for all $i\geq
1$. The labels are of the form $(0/0,1)$, $(1/1,1)$,
$(a/b,p_i)$ or $(a/a,1-p_i)$, for all $i\geq 1$, where $a/b$ is a
label in $\mathcal{T}$ . The labelled edges in
$\mathcal{T}_{\overline{p}}$ are of the form $(T,(x/x,1),T)$ where
$x\in\{0, 1\}$, $(I_{i+1},(a/b,p_i),I_i)$ and $(T,(a/a,1-p_i),I_i)$ where
$(I,a/b,I)$ is a labelled edge in $\mathcal{T}$ or $(T,(a/b,p_i),I_i)$ and $(T,(a/a,1-p_i),I_i)$ where
$(T,a/b,I)$ is a labelled edge in $\mathcal{T}$, for all $i\geq 1$.

The stochastic process $\psi(N)$, with state space $\mathbb{N}$, is
defined by $\psi(N)=\sum_{i=0}^{+\infty}\varepsilon_i(N)F_i$ where
$(\varepsilon_i(N))_{i\geq 0}$ is an infinite sequence of $0$ or $1$
without two $1$ consecutive and with finitely many no zero terms.
The sequence $(\varepsilon_i(N))_{i\geq 0}$ is defined in the
following way:

Put $\varepsilon_i(0)=0$ for all $i$ and assume that we have defined
$(\varepsilon_i(N-1))_{i\geq 0}$ with $N\geq 1$. In the transducer
$\mathcal{T}_{\overline{p}}$, consider the path

\begin{center}
$\ldots (T,(0/0,1),T)\ldots (T,(0/0,1),T)(s_{n+1},(a_n/b_n,t_n),s_n)
\ldots (s_1,(a_0/b_0,t_0),s_0)$
\end{center}

\np where $s_0=I_1$ and $s_{n+1}=T$, such the words $\ldots
00\varepsilon_n(N-1)\ldots \varepsilon_0(N-1)$ and $\ldots
00a_n\ldots a_0$ are equal. We define the sequence
$(\varepsilon_i(N))_{i\geq 0}$ as the infinite sequence whose terms
are $0$ or $1$ such that $\ldots 00\varepsilon_n(N)\ldots
\varepsilon_0(N)=\ldots 00b_n\ldots b_0$.

We remark that $\psi(N-1)$ transitions to $\psi(N)$ with probability $p_{\psi(N-1)\psi(N)}=t_nt_{n-1}\ldots t_0$.

\begin{Example}
If $N=17=100101$ then in the transducer $\mathcal{T}$ of
Fibonacci adding machine, $N$ corresponds to the path $(T,1/1,T) (T,00/01,I)(I,10/00,I)(I,1/0,I)$. In the transducer $\mathcal{T}_{\overline{p}}$ of the stochastic Fibonacci adding machine we have the followings paths:

\np $a)$ $(T,(1/1,1),T)(T,(0/0,1),T) (T,(0/0,1),T)
(T,(1/1,1),T)(T,(0/0,1),T)(T,(1/1,1-p_1),$ $I_1)$. In this case $N=17$
transitions to $17$ with probability $1-p_1$.

\np $b)$ $(T,(1/1,1),T)(T,(0/0,1),T) (T,(0/0,1),T)
(T,(10/10,1-p_2),I_2)(I_2,(1/0,p_1),I_1)$. In this case $N=17$
transitions to $16$ with probability $p_1(1-p_2)$.

\np $c)$ $(T,(1/1,1),T)(T,(00/00,1-p_3),I_3)
(I_3,(10/00,p_2),I_2)(I_2,(1/0,p_1),I_1)$. In this case $N=17$
transitions to $13$ with probability $p_1p_2(1-p_3)$.

\np $d)$ $(T,(1/1,1),T)(T,(00/01,p_3),I_3)
(I_3,(10/00,p_2),I_2)(I_2,(1/0,p_1),I_1)$. In this case $N=17$
transitions to $18$ with probability $p_1p_2p_3$.
\end{Example}

An other way to construct the stochastic Fibonacci adding machine is the same that Killeen and Taylor did on description (\ref{descp}) generalizing the descriptions (\ref{rev1}) and (\ref{rev2}) to
include fallible adding machines by the following manner: let $\overline{p}=(p_i)_{i\geq 1}\subset]0,1]$ be a probabilities sequence and $\{e_i(N):i\geq 0\, \textrm{ } N\in\mathbb{N} \}$ be an
independent, identically distributed family of random variables
parametrized by natural numbers $i,n$ where, for each nonnegative
integer $N$, we have $e_i(N)=0$ with probability  $1-p_{i+1}$ and
$e_i(N)=1$ with probability $p_{i+1}$, for all $i\geq 0$.

Let $N$ a nonnegative integer. Given a sequence
$(\varepsilon_i(N))_{i\geq 0}$ of $0's$ and $1's$ such that
$\varepsilon_i(N)=1$ for finitely many indices $i$ and
$\varepsilon_{i+1}(N)\varepsilon_i(N)\neq 11$, for all $i\geq 0$, we
consider the sequences $(\varepsilon_i(N+1))_{i\geq 0}$ and
$(c_i(N+1))_{i\geq -1}$, defined by $c_{-1}(N+1)=1$ and

$a)$ If $\varepsilon_0(N)=0$ then, for all $i\geq 0$, do

$$
\left\{\begin{array}{l}
\varepsilon_{2i}(N+1)=\left[\frac{\varepsilon_{2i}(N)+e_i(N)c_{i-1}(N+1)}{e_i(N)\varepsilon_{2i+1}(N) c_{i-1}(N+1)+1}\right]; \\
\varepsilon_{2i+1}(N+1)=\left[\frac{\varepsilon_{2i+1}(N)}{e_i(N)c_{i-1}(N+1)+1}\right]; \\
c_i(N+1)=e_i(N)c_{i-1}(N+1)\varepsilon_{2i+1}(N).
\end{array}\right.
$$

$b)$ If $\varepsilon_0(N)=1$ then
$\varepsilon_0(N+1)=\left[\frac{1}{e_0(N)+1}\right]$,
$c_0(N+1)=e_0(N)$ and, for all $i\geq 1$, do

$$
\left\{\begin{array}{l}
\varepsilon_{2i-1}(N+1)=\left[\frac{\varepsilon_{2i-1}(N)+e_i(N)c_{i-1}(N+1)}{e_i(N)\varepsilon_{2i}(N) c_{i-1}(N+1)+1}\right]; \\
\varepsilon_{2i}(N+1)=\left[\frac{\varepsilon_{2i}(N)}{e_i(N)c_{i-1}(N+1)+1}\right]; \\
c_i(N+1)=e_i(N)c_{i-1}(N+1)\varepsilon_{2i}(N).
\end{array}\right.
$$

The systems $a)$ and $b)$ gives to us the Fibonacci stochastic adding
machine.

\begin{figure}[!h]
\centering
\includegraphics[scale=0.5]{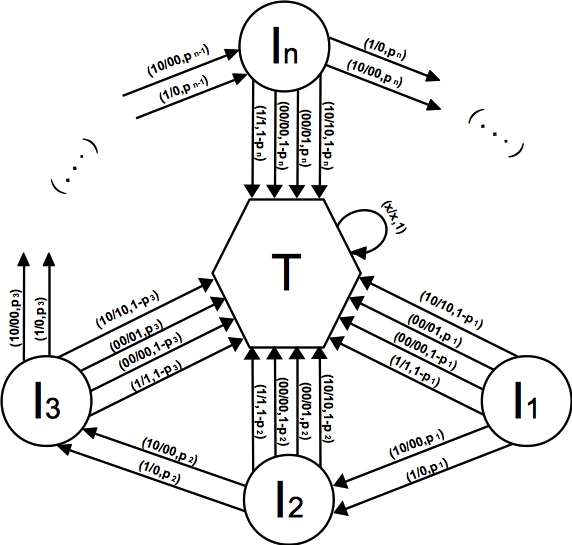}
\caption{Transducer $\mathcal{T}_{\overline{p}}$ of the stochastic
Fibonacci adding machine. } \label{maqpi}
\end{figure}

With the transition probabilities, we obtain the countable
transition matrix of the Markov chain
$S=S_{\overline{p}}=(S_{i,j})_{i,j\geq 0}$. This Markov chain is irreducible and aperiodic.
To help the reader, the first entries of the matrix $S$ are given by table \ref{matrizspi}.

\begin{figure}[!h]
\centering
\includegraphics[scale=0.435]{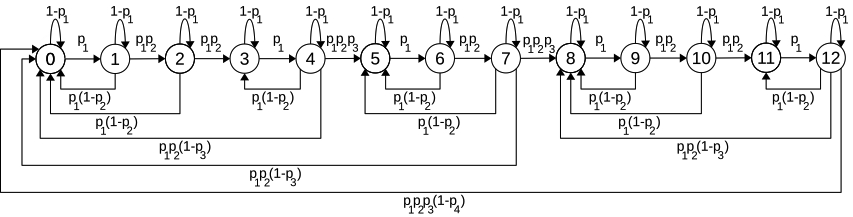}
\caption{Transition graph of the stochastic Fibonacci adding
machine. } \label{grafo}
\end{figure}

\begin{table}[htb]
\centering \tiny
$$
\left(
\begin{tabular}{ccccccccccl}
 $1-p_1$ & $p_1$ & 0 & 0 & 0 & 0 & 0 & 0 & 0 & 0 & $\ldots$  \\
 $p_1(1-p_2)$ & $1-p_1$ & $p_1p_2$ & 0 & 0 & 0 & 0 & 0 & 0 & 0 & $\ldots$  \\
 $p_1(1-p_2)$ & 0 & $1-p_1$ & $p_1p_2$ & 0 & 0 & 0 & 0 & 0 & 0 & $\ldots$  \\
 0 & 0 & 0 & $1-p_1$ & $p_1$ & 0 & 0 & 0 & 0 & 0 & $\ldots$  \\
 $p_1p_2(1-p_3)$ & 0 & 0 & $p_1(1-p_2)$ & $1-p_1$ & $p_1p_2p_3$ & 0 & 0 & 0 & 0 & $\ldots$ \\
 0 & 0 & 0 & 0 & 0 & $1-p_1$ & $p_1$ & 0 & 0 & 0 & $\ldots$ \\
 0 & 0 & 0 & 0 & 0 & $p_1(1-p_2)$ & $1-p_1$ & $p_1p_2$ & 0 & 0 & $\ldots$ \\
 $p_1p_2(1-p_3)$ & 0 & 0 & 0 & 0 & $p_1(1-p_2)$ & 0 & $1-p_1$ & $p_1p_2p_3$ & 0 & $\ldots$ \\
 0 & 0 & 0 & 0 & 0 & 0 & 0 & 0 & $1-p_1$ & $p_1$ & $\ldots$ \\
 0 & 0 & 0 & 0 & 0 & 0 & 0 & 0 & $p_1(1-p_2)$ & $1-p_1$ & $\ldots$  \\
 0 & 0 & 0 & 0 & 0 & 0 & 0 & 0 & $p_1(1-p_2)$ & 0 & $\ldots$ \\
  0 & 0 & 0 & 0 & 0 & 0 & 0 & 0 & 0 & 0 & $\ldots$  \\
 $p_1p_2p_3(1-p_4)$ & 0 & 0 & 0 & 0 & 0 & 0 & 0 & $p_1p_2(1-p_3)$ & 0 & $\ldots$ \\
  0 & 0 & 0 & 0 & 0 & 0 & 0 & 0 & 0 & 0 & $\ldots$  \\
  \vdots & \vdots & \vdots & \vdots & \vdots & \vdots & \vdots & \vdots & \vdots & \vdots & $\ddots$  \\
\end{tabular}
\right)
$$
\caption{First entries of the transition operator $S$.}
\label{matrizspi}
\end{table}

\begin{Proposition}
Let $N$ be a nonnegative integer. Then the following results are
satisfied.

\np $\textbf{i)}$  $N$ transitions to $N$ with probability $1-p_1$.

\np $\textbf{ii)}$ If $N=\varepsilon_k\ldots\varepsilon_0$ with
$\varepsilon_1\varepsilon_0=00$ then $N$ transitions to $N+1$ with
probability $p_1$.

\np $\textbf{iii)}$ If $N=\varepsilon_k\ldots\varepsilon_t(10)^s$
with $\varepsilon_{t+1}\varepsilon_t=00$, $s\geq 1$ and $t=2s$ then
$N$ transitions to $N+1$ with probability $p_1p_2\ldots p_{s+1}$ and
$N$ transitions to $N-F_{2m}+1$, $0< m\leq s$, with probability
$p_1p_2\ldots p_m(1-p_{m+1})$.

\np $\textbf{iv)}$ If $N=\varepsilon_k\ldots\varepsilon_t(10)^s1$
with $\varepsilon_{t+1}\varepsilon_t=00$, $s\geq 1$ and $t= 2s+1$
then $N$ transitions to $N+1$ with probability $p_1p_2\ldots
p_{s+2}$ and $N$ transitions to $N-F_{2m-1}+1$, $0<m\leq s+1$, with
probability $p_1p_2\ldots p_m(1-p_{m+1})$. \label{prop431}
\end{Proposition}
\begin{proof}
\np \textbf{(ii)} If $N=\varepsilon_k\ldots\varepsilon_0$ with
$\varepsilon_1\varepsilon_0=00$ then
$(T,(\varepsilon_k/\varepsilon_k,1),T)\ldots
(T,(\varepsilon_2/\varepsilon_2,1),T)(T,(00/$ $01,p_1),I_1)$
is a path in the transducer $\mathcal{T}_{\overline{p}}$. Then
$N$ transitions to $N+1$ with probability $p_1$.

\np $\textbf{(iii)}$ If $N=\varepsilon_k\ldots\varepsilon_t(10)^s$
then $\varepsilon_{t+1}\varepsilon_t=00$, $s\geq 1$ and $t= 2s$.
Therefore,

\np $a)$ $(T,(\varepsilon_k/\varepsilon_k,1),T)\ldots
(T,(\varepsilon_{t+2}/\varepsilon_{t+2},1),T)(T,(00/01,p_{s+1}),I_{s+1})$

$\underbrace{(I_{s+1},(10/00,p_s),I_s)\ldots(I_3,(10/00,p_2),I_2)(I_2,(10/00,p_1),I_1)}_s$ is a path in the transducer $\mathcal{T}_{\overline{p}}$. In this case, $N$ transitions to $N+1$ with probability $p_1\ldots
p_sp_{s+1}$.

\np $b)$
$(T,(\varepsilon_k/\varepsilon_k,1),T)\ldots(T,(\varepsilon_t/\varepsilon_t,1),T)$

$\underbrace{(T,(1/1,1),T)(T,(0/0,1),T) \ldots
(T,(1/1,1),T)(T,(0/0,1),T)}_{2s-2m-2}(T,(10/10,1-p_{m+1}),I_{m+1})$

$(I_{m+1},(10/00,p_m),I_m)\ldots(I_3,(10/00,p_2),I_2)(I_2,(10/00,p_1),I_1)$ is a path in the transducer $\mathcal{T}_{\overline{p}}$. In this case, $N$ transitions to $N-F_{2m}+1$ with probability
$p_1\ldots p_m(1-p_{m+1})$.

\np The cases $(i)$ and $(iv)$ are left to the reader.
\end{proof}

\begin{Proposition}
Let $i,j\in\mathbb{N}^*$ such that $F_n\leq i,j\leq F_{n+1}-1$. Then
the transition probability from $i$ to $j$ is equal to transition
probability from $i-F_n$ to $j-F_n$, i.e. $S_{i,j}=S_{i-F_n,j-F_n}$.
\label{prop432}
\end{Proposition}
\begin{proof} If $i=j$ then by proposition \ref{prop431} we
have that $S_{i,j}=S_{i-F_n,j-F_n}=1-p_1$.

Suppose $i\neq j$ and $S_{i,j}>0$.

$i)$ If $i=\varepsilon_n\ldots\varepsilon_200$ with
$\varepsilon_n=1$ then by the item $(ii)$ of the proposition
\ref{prop431}, $j=i+1=\varepsilon_n\ldots\varepsilon_201$ and
$S_{i,j}=p_1$. Thus $i-F_n=\varepsilon_{n-1}\ldots\varepsilon_200$,
$j-F_n=\varepsilon_{k-1}\ldots\varepsilon_201$ and
$S_{i-F_n,j-F_n}=p_1=S_{i,j}$.

$ii)$ If
$i=\varepsilon_n\ldots\varepsilon_t00\underbrace{1010\ldots1010}_{2s}$,
with $s\geq 1$ then by the item $(iii)$ of the proposition
\ref{prop431}, we have two cases:
$j=\varepsilon_n\ldots\varepsilon_t01\underbrace{0000\ldots0000}_{2s}$
and $S_{i,j}=p_1p_2\ldots p_{s+1}$ or
$j=i-F_{2m}+1=\varepsilon_n\ldots\varepsilon_t00\underbrace{10\ldots10}_{2s-2m}\underbrace{00\ldots00}_{2m}$,
with $1\leq m\leq s$ e $S_{i,j}=p_1p_2\ldots p_m(1-p_{m+1})$. Since
$F_n\leq i,j\leq F_{n+1}-1$ then the Fibonacci expansions of $i$
and $j$ are $i=1c_{n-1}\ldots c_0$ and $j=1d_{n-1}\ldots d_0$, where
$c_l,d_l\in\{0,1\}$, $c_{n-1}=d_{n-1}=0$, $c_{l+1}c_l\neq 11$ and
$d_{l+1}d_l\neq 11$, for all $l\in \{0,\ldots,n-2\}$. Therefore, in
the first case there exists a integer $t\leq n_0\leq n-2$ such that
$\varepsilon_{n_0}=1$ and $\varepsilon_l=0$ for all integer
$n_0<l\leq n-1$. Hence, we have $i-F_n=\varepsilon_{n_0}\ldots
\varepsilon_t00\underbrace{1010\ldots 1010}_{2s}$,
$j-F_n=\varepsilon_{n_0}\ldots \varepsilon_t01\underbrace{0000\ldots
0000}_{2s}$ and $S_{i-F_n,j-F_n}=p_1p_2\ldots p_{s+1}=S_{i,j}$.

\np In the second case there exists a integer $n_0$ such that $n_0=
2s-1$ or $t\leq n_0\leq n-2$ such that $\varepsilon_{n_0}=1$ and
$\varepsilon_l=0$ for all $n_0<l\leq n-1$. Hence, we have
$i-F_n=\varepsilon_{n_0}\ldots \varepsilon_t00\underbrace{1010\ldots
1010}_{2s}$ and $j-F_n=\varepsilon_{n_0}\ldots
\varepsilon_t00\underbrace{10\ldots10}_{2s-2m}\underbrace{00\ldots
00}_{2m}$ (case $t\leq n_0\leq n-2$) or
$i-F_n=00\underbrace{1010\ldots 1010}_{2s-2}$ and
$j-F_n=00\underbrace{10\ldots10}_{2s-2-2m}\underbrace{00\ldots
00}_{2m}$ (case $n_0=2s-1$). Therefore,
$S_{i-F_n,j-F_n}=p_1p_2\ldots p_m(1-p_{m+1})=S_{i,j}$.

$iii)$ Using the same idea used in the proof of the anterior item,
we deduce the case that $N=\varepsilon_n\ldots\varepsilon_t
00\underbrace{1010\ldots1010}_{2s}1$.
\end{proof}

\begin{Proposition}
Let $i,j,n\in\mathbb{N}^*$ such that $0<i<F_n\leq j$. Then
$S_{j,i}=0$. \label{prop433}
\end{Proposition}
\begin{proof} Let $n\in\mathbb{N}$ such that $0<i<F_n\leq j<F_{n+1}$. Hence,
the expressions of $i$ and $j$ in Fibonacci base are
\begin{equation}
i=\varepsilon_{n-1}'\ldots\varepsilon_0' \textrm{    and    }
j=1\varepsilon_{n-1}\ldots\varepsilon_0, \label{rel41}
\end{equation}
\np where $\varepsilon_l,\varepsilon_l'\in\{0,1\}$,
$\varepsilon_{l+1}\varepsilon_l\neq 11$ and
$\varepsilon_{l+1}'\varepsilon_l'\neq 11$ for all $l\in\{0,\ldots,
n\}$. If $S_{j,i}\neq 0$ then from proposition \ref{prop431} and
the fact that $i<j$, we deduce that
$j=\varepsilon_n\ldots\varepsilon_t00\underbrace{10\ldots 10}_{2s}$,
with $s\geq 1$ and $n\geq t \geq 2s+2$, and
$i=j-F_{2m}+1=\varepsilon_n\ldots\varepsilon_t00\underbrace{10\ldots
10}_{2s-2m}\underbrace{00\ldots 00}_{2m}$, with $1\leq m\leq s$, or
$j=\varepsilon_n\ldots\varepsilon_t00\underbrace{10\ldots10}_{2s-2}1$,
with $s\geq 2$ and $n\geq t\geq 2s+1$, and
$i=j-F_{2m+1}+1=\varepsilon_n\ldots\varepsilon_t00\underbrace{10\ldots
10}_{2s-2m}\underbrace{00\ldots 000}_{2m-1}$, with $1\leq m\leq s$.

In both cases we obtain a contraction because of the relation
(\ref{rel41}).
\end{proof}

\section{Probabilistic properties of the stochastic Fibonacci adding machine}

In the next proposition, we obtain a necessary and sufficient
condition for transience of the Markov chain.

\begin{Proposition}
The Markov chain is transient if only if $\prod_{i=1}^{\infty}p_i>0$. \label{marktrans}
\end{Proposition}

For the proof of the proposition \ref{marktrans}, we need the
following definition and lemma:

\begin{Lemma}\cite{o}
An irreducible and aperiodic Markov chain is transient if only if there exists a sequence
$v=(v_i)_{i\geq 1}$ such that $0<v_i\leq 1$ and $v_i=\sum_{j=1}^{+\infty}S_{i,j}v_j$ for all $i\geq 1$, i.e. $\tilde{S}v=v$ where $\tilde{S}$ is
obtained from $S$ by removing its first line and column. \label{lemaclassico}
\end{Lemma}

Indeed, in the transient case a solution is obtained by taking $v_m$ as the
probability that $0$ is never visited by the Markov chain given that
the chain starts at state $m$, see the discussion on pp. $42-43$,
Chap. $2$ in \cite{la} and also Corollary $16.48$ in \cite{o}.

\begin{Definition}
Let $\beta:\mathbb{N}^*\longrightarrow\mathbb{R_+}$ be the map defined
by $\beta(n)=\Pi_r$ if $F_r\leq n< F_{r+1}$, for all $r\geq 0$,
where $\Pi_0=1$, $\Pi_1=\frac{1}{p_2}$ and
$\Pi_r:=\displaystyle\prod_{i=2}^{\left[\frac{r}{2}\right]+1}\dfrac{1}{p_i^2}\dfrac{1}{p^{
\left[\frac{r+1}{2}\right]-\left[\frac{r}{2}
\right]}_{\left[\frac{r}{2}\right]+2}}$,
for all $r\geq 2$, i.e.

\begin{center}
$\Pi_{2n}=\displaystyle\prod_{i=2}^{n+1}\frac{1}{p_i^2}=\frac{1}{p_2^2p_3^2\ldots
p_{n+1}^2}$ and
$\Pi_{2n+1}=\frac{1}{p_{n+2}}\displaystyle\prod_{i=2}^{n+1}\frac{1}{p_i^2}=\frac{1}{p_2^2p_3^2\ldots
p_{n+1}^2} \frac{1}{p_{n+2}}$, for all $n\geq 1$.
\end{center}
\label{beta}
\end{Definition}

\begin{Lemma}
Let $v=(v_i)_{i\geq 1}\in l^\infty(\mathbb{N})$ be a bounded sequence.
Then $\tilde{S}v=v$ if only if
$v_n=\beta(n)v_1$ for all $n\geq 1$, where $\tilde{S}$ is the
matrix obtained from $S$ by removing its first line and column.
\label{lemamark}
\end{Lemma}
\begin{proof}
Assume that $\tilde{S}v=v$. Since  $S_{i,i+k}=0$ for all integer $i\geq 0$ and $k\geq2$,
(see the proposition \ref{prop431}) then the
operator $\tilde{S}$ satisfies $\tilde{S}_{i,i+k}=S_{i,i+k}=0$, for
all integers $i\geq 1$ and $k\geq 2$. Therefore, it is
possible to prove by induction that for each integer
$n\geq 1$, there exists $\alpha_n\in\mathbb{R}$ such that
$v_n=\alpha_n v_1$.

\np \textbf{Claim:} $\alpha_n=\beta(n)$, for all integer $n\geq 1$.

Indeed, since $(\tilde{S}-I)v=0$, it's not hard to see that 
$\alpha_1=1=\beta(1)$, $\alpha_2=\frac{1}{p_2}=\beta(2)$
and $\alpha_3=\alpha_4=\frac{1}{p_2^2}=\beta(3)=\beta(4)$.

Let $k\in\mathbb{N}$, $k\geq 5$, and suppose that
$\alpha_n=\beta(n)$, for all $n\in\{1,\ldots,k\}$.

To prove $\alpha_{k+1}=\beta(k+1)$, we need to consider two
cases:

\np \textbf{(i)} $k+1= F_r$ for some $r\geq 4$;

\np In this case we can have $r=2l+1$ or $r=2l+2$, for some $l\geq
1$. In both cases we have

$$
\left\{
\begin{array}{l}
S_{k,k+1}=p_1p_2\ldots p_{l+1}p_{l+2}; \\
S_{k,0}=p_1p_2\ldots p_{l+1}(1-p_{l+2}); \\
S_{k,\sum_{i=0}^jF_{r-1-2i}}= p_1p_2\ldots p_{l-j}(1-p_{l-j+1}), \textrm{ for all } j=0\ldots,l-1; \\
S_{k,k}=1-p_1. \\
\end{array}
\right.
$$

\np Since $(\tilde{S}-Iv)_k=0$, we have
\begin{equation}
p_1p_2\ldots p_{l+1}p_{l+2}v_{k+1}+
\displaystyle\sum_{j=0}^{l-1}p_1p_2\ldots
p_{l-j}(1-p_{l-j+1})v_{\sum_{i=0}^jF_{r-1-2i}}+(1-p_1)v_k=v_k.
\label{kfr}
\end{equation}
\np Thus, since
$F_{r-1}\leq\sum_{i=0}^jF_{r-1-2i}\leq \sum_{i=0}^{l-1}F_{r-1-2i}<F_r=k+1$, for all $j\in\{0,\ldots,l-1\}$ and
$\alpha_n=\beta(n)$, for all $n\in\{1,\ldots, k\}$ then $v_{\sum_{i=0}^jF_{r-1-2i}}=v_k=\Pi_{r-1}v_1$, for all
$j=0,\ldots,l-1$.

\np Therefore, from relation (\ref{kfr}), we obtained

\begin{center}
$p_1p_2\ldots p_{l+1}p_{l+2}v_{k+1}+
\displaystyle\sum_{j=0}^{l-1}p_1p_2\ldots
p_{l-j}(1-p_{l-j+1})\Pi_{r-1}v_1+(1-p_1)\Pi_{r-1}v_1=\Pi_{r-1}v_1$.
\end{center}

Then it implies that $p_1p_2\ldots p_{l+1}p_{l+2}v_{k+1}= p_1p_2\ldots
p_{l+1}\Pi_{r-1}v_1$ i.e.

\begin{center}
$v_{k+1}=\dfrac{1}{p_{l+2}}\Pi_{r-1}v_1=\dfrac{1}{p_{l+2}}
\displaystyle\prod_{i=2}^{\left[\frac{r-1}{2}\right]+1}\dfrac{1}{p_i^2}\dfrac{1}{p^{
\left[\frac{r-1+1}{2}\right]-\left[\frac{r-1}{2}\right]}_{\left[\frac{r-1}{2}\right]+2}}
v_1$.
\end{center}

If $r=2l+1$ then $v_{k+1}=
\displaystyle \prod_{i=2}^{l+1}\dfrac{1}{p_i^2}\dfrac{1}{p_{l+2}}v_1=\Pi_rv_1$ i.e. $\alpha_{k+1}=\beta(k+1)$.

If $r=2l+2$ then $v_{k+1}=
\displaystyle\prod_{i=2}^{l+2}\dfrac{1}{p_i^2}v_1=\Pi_rv_1$ i.e. $\alpha_{k+1}=\beta(k+1)$.

\np \textbf{(ii)} $F_r< k+1< F_{r+1}$, for some $r\geq 3$;

\np In this case, we have that $F_r\leq k< F_{r+1}-1$. Then by proposition \ref{prop433} it implies
that $S_{k,i}=0$, for all $i\in\{0,\ldots,F_r-1\}$. Thus, we have that

\begin{center}
$1=\sum_{i=0}^{k+1}S_{k,i}=\sum_{i=F_r}^{k+1}S_{k,i}=\sum_{i=F_r}^{k+1}\tilde{S}_{k,i}$ and $1-S_{k,k+1}=\sum_{i=F_r}^k\tilde{S}_{k,i}$.
\end{center}

\np Therefore, since $(\tilde{S}-I)v=0$ and
$\alpha_n=\beta(n)=\Pi_r$, for all $n\in\{F_r,\ldots, k\}$, we have

\begin{center}
$\Pi_rv_1= v_k= \sum_{i=1}^{k+1}\tilde{S}_{k,i}v_i=
\sum_{i=F_r}^{k}\tilde{S}_{k,i}\Pi_rv_1+\tilde{S}_{k,k+1}v_{k+1}=(1-\tilde{S}_{k,k+1})\Pi_rv_1+\tilde{S}_{k,k+1}v_{k+1}$.
\end{center}

Thus, it follows that
$\tilde{S}_{k,k+1}v_{k+1}=\tilde{S}_{k,k+1}\Pi_rv_1$, i.e. $\alpha_{k+1}=\Pi_r=\beta(k+1)$.
\end{proof}

\np \textbf{Proof of the proposition \ref{marktrans}:} 
Suppose that the Markov chain is transient. Then there exists $v=(v_i)_{i\geq 1}$
satisfying the conditions of the lemma \ref{lemaclassico}. By lemma \ref{lemamark}, we have
that $v_n=\beta(n)v_1$, for all $n\geq 1$. Thus, we need to have $\prod_{i=1}^{+\infty}p_i>0$, because otherwise we should have
$v_n=\beta(n)v_1$ goes to infinity, contradicting the
fact that $|v_i|\leq 1$ for all $i\geq 1$.
\rightline{$\Box$}

\begin{Proposition}
Let $\mu=(\mu_i)_{i\geq 0}$ be a sequence on $\mathbb{C}$.
Then $\mu S=\mu$ if only if there exists a sequence $(\xi_i)_{i\geq 1}=(\xi_i(\overline{p}))_{i\geq 1}$ such that $\mu_i=\xi_i\mu_0$, for all integer $i\geq
1$, where $\xi_1=1$, $\xi_2=p_2$ and $\xi_{F_n+k}=\xi_{F_n}\xi_k$, for all integer $n\geq 2$ and
$k\in\{1,\ldots,F_{n-1}-1\}$. Furthermore, $\xi_{F_n}=p_{\left[\frac{n+1}{2}\right]+1}$ and
$\xi_{F_n-1}=p_2p_3\ldots
p_{\left[\frac{n}{2}\right]+1}$. \label{marknula}
\end{Proposition}
\begin{proof}
Assume that $\mu S=\mu$. Since for each nonnegative integer $n\geq 1$ we have $S_{i,j}=0$, with $0<j<F_n\leq i$ (see proposition \ref{prop433}) then $\mu_j=\sum_{i=0}^{+\infty}\mu_i
S_{i,j}=\sum_{i=0}^{F_n-1}\mu_i S_{i,j}$, for all integer $n\geq 1$ and
$j\in\{1,\ldots,F_n-1\}$. Then it's possible to prove by induction
that there exists $(\xi_j)_{j\geq1}=(\xi_j(\overline{p}))_{j\geq1}\in l^\infty(\mathbb{N})$ such
that $\mu_j=\xi_j \mu_0$, for all $j\geq 1$. Furthermore, it's not hard to prove that
$\xi_1=1$, $\xi_2=p_2$, $\xi_{F_2}=\xi_3=p_2=p_{\left[\frac{2+1}{2}\right]+1}$
and $\xi_{F_2+1}=\xi_{F_3-1}=\xi_4=p_2=\xi_1\xi_{F_2}$.

Let $n\in\mathbb{N}$, $n\geq 3$ and $k\in\{1,\ldots,F_{n-1}-1\}$. Since $S_{i,F_n+k}=0$, for each $i\in\{0,\ldots, F_n+k-2\}$, we have that $\mu_{F_n+k}=\sum_{i=0}^{+\infty}\mu_iS_{i,F_n+k}=\sum_{i=F_n+k-1}^{+\infty}\mu_iS_{i,F_n+k}$.

Then from proposition \ref{prop432}, we have that

\begin{center}
$\mu_{F_n+k}=\displaystyle\sum_{i=F_n+k-1}^{+\infty}\mu_iS_{i,F_n+k}=
\sum_{i=F_n+k-1}^{+\infty}\mu_iS_{i-F_n,k}=
\sum_{i=k-1}^{+\infty}\mu_{i+F_n}S_{i,k}=
\sum_{i=0}^{+\infty}\mu_{i+F_n}S_{i,k}$.
\end{center}

Let $(\mu'_i)_{i\geq 0}$ defined by $\mu'_i=\mu_{F_n+i}$. Hence $\mu'_k=\sum_{i=0}^{+\infty}\mu'_iS_{i,k}=\sum_{i=0}^{F_n-1}\mu'_iS_{i,k}$, for all $k\in\{1,\ldots,
F_{n-1}-1\}$.

Therefore, $\mu'_k=\xi_k\mu'_0$, i.e. $\mu_{F_n+k}=\xi_k\mu_{F_n}$.
Since $\mu_{F_n}=\xi_{F_n}\mu_0$ and $\mu_{F_n+k}=\xi_{F_n+k}\mu_0$, it follows that
$\xi_{F_n+k}=\xi_{F_n}\xi_k$, for all $k\in\{1,\ldots, F_{n-1}-1\}$.

Now we have that $\xi_{F_n}=p_{\left[\frac{n+1}{2}\right]+1}$ and
$\xi_{F_m-1}=p_2p_3\ldots p_{\left[\frac{m}{2}\right]+1}$ is true for $n=1,2$ and $m=2,3$. Let $k>2$ a integer number and supposes that $\xi_{F_n}=p_{\left[\frac{n+1}{2}\right]+1}$ and $\xi_{F_m-1}=p_2p_3\ldots p_{\left[\frac{m}{2}\right]+1}$, for all $n\in\{1,\ldots,k-1\}$ and $m\in\{2,\ldots,k\}$.
Assume $\mu=\mu S$. Then

\begin{center}
$\mu_{F_k}=\displaystyle\sum_{i=0}^{+\infty}\mu_i
S_{i,F_k}=\sum_{i=F_k-1}^{+\infty}\mu_i
S_{i,F_k}=\mu_{F_k-1}S_{F_k-1,F_k}+\mu_{F_k}S_{F_k,F_k}+\sum_{i=F_k+1}^{F_k+F_{k-1}-1}\mu_iS_{i,F_k}$,
\end{center}

On the other hand, by the first part of proposition, we have that

\begin{center}
$\displaystyle\sum_{i=F_k+1}^{F_k+F_{k-1}-1}\mu_iS_{i,F_k}= \sum_{i=F_k+1}^{F_k+F_{k-1}-1}\mu_iS_{i-F_k,0}=\sum_{i=1}^{F_{k-1}-1}\mu_{F_k+i}S_{i,0}= \left(\sum_{i=1}^{F_{k-1}-1}\xi_iS_{i,0} \right)\mu_{F_k}= \left(\displaystyle\sum_{i=1}^{n-1}\xi_{F_i-1}S_{F_i-1,0}
\right)\mu_{F_k}$.
\end{center}

\np Thus, from proposition \ref{prop431}, it follows that

\begin{center}
$\mu_{F_k}= p_2p_3\ldots
p_{\left[\frac{k}{2}\right]+1}\mu_0p_1p_2\ldots
p_{\left[\frac{k+1}{2}\right]+1} +(1-p_1)\mu_{F_k}+$
\end{center}
\begin{center}
$+\left(p_1(1-p_2)+p_2p_1(1-p_2)+p_2p_1p_2(1-p_3)+p_2p_3
p_1p_2(1-p_3)+p_2p_3p_1p_2p_3(1-p_4)+\right.$
\end{center}
\begin{center}
$\left. +\ldots + p_2p_3\ldots p_{\left[\frac{k-1}{2}\right]+1}p_1p_2\ldots
p_{\left[\frac{k}{2}\right]}
\left(1-p_{\left[\frac{k}{2}\right]+1}\right)\right)\mu_{F_k}$.
\end{center}

\np Hence, it follows that $p_2p_3\ldots p_{\left[\frac{k-1}{2}\right]+1}\mu_{F_k}=p_2p_3\ldots
p_{\left[\frac{k+1}{2}\right]+1}\mu_0$, i.e.
$\mu_{F_k}=p_{\left[\frac{k+1}{2}\right]+1}\mu_0$.

On the other hand, we have that

\begin{center}
$\xi_{F_{k+1}-1}=\xi_{F_k+F_{k-1}-1}=\xi_{F_{k-1}-1}\xi_{F_k}=p_2p_3\ldots
p_{\left[\frac{k-1}{2}\right]+1}
p_{\left[\frac{k+1}{2}\right]+1}=p_2p_3\ldots
p_{\left[\frac{k+1}{2}\right]+1}$.
\end{center}
\end{proof}

\begin{Theorem}
If $\prod_{i=1}^{+\infty}p_i=0$ and $\sum_{i=1}^{+\infty}p_i=+\infty$ then the Markov
chains is null recurrent.
\end{Theorem}
\begin{proof}
Since $\prod_{i=1}^{+\infty}p_i=0$ then by the theorem
\ref{marktrans} we have that the Markov chain is recurrent.

Let $\mu=(\mu_0,\mu_1,\mu_2,\ldots)\in l^1(\mathbb{N})$ be a
invariant measure for the operator $S$. Then $\mu S=\mu$. Hence,
from proposition \ref{marknula}, it follows that $\|\mu\|_1=\displaystyle\sum_{i=0}^{+\infty}|\mu_i|\geq
\sum_{i=1}^{+\infty}|\mu_{F_i}|=\sum_{i=2}^{+\infty}2p_i= +\infty$, which yields a contradiction to the fact that $\mu\in
l^1(\mathbb{N})$. Hence, if $\sum_{i=1}^{+\infty}p_i=+\infty$ then $S$ has no
invariant probability measure and so cannot be positive recurrent.

Therefore, the Markov chain is null recurrent.
\end{proof}

\begin{Theorem}
There exists a sequence of probabilities $\overline{p}=(p_i)_{i\geq 1}$
such that satisfies the Markov chain to be positive recurrent. In particular, if $\sum_{i=2}^{+\infty}p_iF_{2(i-1)}<+\infty$ then the Markov chain is positive recurrent.
\label{recpositiva}
\end{Theorem}
\begin{proof} We will construct a probabilities sequence $(p_i)_{i\geq 1}$,
like the following:

Let $p_1,p_2\in]0,1]$ and $(b_n)_{n\geq 0}\in l^1(\mathbb{N})$, with
$b_1=1$ and $b_2=3p_2$.

Let $(a_n)_{n\geq 0}$ be a sequence of positive real number, where
$a_0=1=b_1$, $a_1=p_2$, $a_2=2p_2$
and for each $k\geq 3$, let $0<p_k<1$ such that $p_k\left(2\displaystyle\sum_{i=0}^{2(k-2)-3}a_i+2+a_{2(k-2)-2}\right)\leq
b_k$ and define $a_{2(k-1)-1}:=p_k\displaystyle\sum_{i=0}^{2(k-2)-3}a_i+p_k$ and
$a_{2(k-1)}:=p_k\displaystyle\sum_{i=0}^{2(k-2)-2}a_i+p_k$.

Construct the Fibonacci stochastic adding machine related to the
probability sequence $(p_i)_{i\geq 1}$ and let $S$ be the transition operator and $\mu\in l^\infty(\mathbb{N})$
such that $\mu S=\mu$. Thus, from proposition \ref{marknula},
$\mu_i=\xi_i\mu_0$, for all $i\geq 0$, where $\xi_0=1$,
$\xi_1=1$, $\xi_2=p_2$ and for each $n\geq 2$ we have
$\xi_{F_n+k}=\xi_{F_n}\xi_k$, for each
$k\in\{1,\ldots,F_{n-1}-1\}$, and
$\xi_{F_n}=p_{\left[\frac{n+1}{2}\right]+1}$.

Define the sequence $(\alpha_n)_{n\geq 0}$ by
$\alpha_n=\displaystyle\sum_{i=F_n}^{F_{n+1}-1}\xi_i$ and note
that, from proposition \ref{marknula}, for each $n\geq 3$, we have

\begin{center}
$\alpha_n=\displaystyle
\sum_{i=F_n}^{F_{n+1}-1}\xi_i=\sum_{i=0}^{F_{n-1}-1}\xi_{F_n+i}=
\xi_{F_n}\left(\sum_{i=1}^{F_{n-1}-1}\xi_i+\xi_0\right)=
\xi_{F_n}\left(\sum_{i=0}^{n-2}\alpha_i+1\right)$.
\end{center}

It's possible to prove by induction that $\alpha_n=a_n$, for all
$n\geq 0$. Hence,

\begin{center}
$\|\mu\|_1=|\mu_0|\displaystyle\sum_{i=0}^{+\infty}\xi_i=
|\mu_0|\xi_0+|\mu_0|\sum_{i=0}^{+\infty}a_i =|\mu_0|+
|\mu_0|\left(a_0+a_1+a_2+\sum_{k=3}^{+\infty}(a_{2(k-1)-1}+a_{2(k-1)})\right)\leq
|\mu_0|+
|\mu_0|\left(b_1+b_2+\sum_{k=3}^{+\infty}b_k\right)<+\infty$.
\end{center}

Therefore, $\mu\in l^1(\mathbb{N})$ and $\mu$ is a invariant measure
for the operator $S$, i.e. the Markov chain is positive recurrent. On the other hand, we have
\begin{equation}
\|\mu\|_1=\displaystyle |\mu_0|+|\mu_0|\sum_{i=0}^{+\infty}\alpha_i=
2|\mu_0|+|\mu_0|\sum_{k=2}^{+\infty}(\alpha_{2(k-1)-1}+
\alpha_{2(k-1)}) \label{saci}
\end{equation}
\np and $\alpha_n=\xi_{F_n}\displaystyle \sum_{i=0}^{F_{n-1}-1}\xi_i$
for all $n\geq 1$. From proposition \ref{marknula}, we have that
$0<\xi_i\leq 1$, for all $i\geq 0$. Wherefore,
$\alpha_n\leq\xi_{F_n}F_{n-1}$, for all $n\geq 1$. Hence, we have from relation (\ref{saci}) that
$\|\mu\|_1\leq
2|\mu_0|+|\mu_0|\sum_{k=2}^{+\infty}p_k(F_{2(k-1)-2}+
F_{2(k-1)-1})=2|\mu_0|+|\mu_0|\sum_{k=2}^{+\infty}p_kF_{2(k-1)}$.

Therefore, if $\sum_{k=2}^{+\infty}p_kF_{2(k-1)}<+\infty$ then $\mu\in l^1(\mathbb{N})$ and the Markov chain is
positive recurrent.
\end{proof}

\section{Spectral properties of the transition operator}

\begin{Definition}
Let $X$ be a Banach space over $\mathbb{C}$ and $T:X\longrightarrow X$ be a linear operator. Then we have the following definitions:

\np \textbf{a)} \textit{Spectrum of $T$: } $\sigma (T)=\{\lambda\in\mathbb{C}:T-\lambda I \textrm{ is not bijective} \}$.

\np \textbf{b)} \textit{Point spectrum of $T$:} $\sigma_{pt} (T)=\{\lambda\in\mathbb{C}:T-\lambda I \textrm{ is not one-to-one} \}$.

\np \textbf{c)} \textit{Approximate point spectrum of $T$: } $\sigma_a(T)=\{\lambda\in\mathbb{C}: \textrm{there exists a sequence } (x_n)_{n\geq 0}$ in $X, \textrm{ with } \|x_n\|=1 \textrm{ for all } n\geq 0 \textrm{ and } \lim_{n\to\infty}\| (\lambda I-T)x_n\|=0 \}$.

\np \textbf{d)} \textit{Continuous spectrum of $T$: } $\sigma_c(T)=\{\lambda\in\mathbb{C}: \lambda I-T \textrm{ is one-to-one, } \overline{(\lambda I-T)X}=X \textrm{ and } (\lambda I-T)X\neq X \}$.

\np \textbf{e)} \textit{Residual spectrum of $T$: } $\sigma_r(T)=\{\lambda\in\mathbb{C}: \lambda I-T \textrm{ is one-to-one and } \overline{(\lambda I-T)X}\neq X \}$.
\end{Definition}

Hence, we have $\sigma(T)=\sigma_{pt}(T)\cup\sigma_c(T)\cup \sigma_r(T)$ and $\sigma_a(T)\subset\sigma(T)$.

\begin{Definition} For each $\lambda\in\mathbb{C}$, let $(q_n(\lambda))_{n\geq
1}=(q_n)_{n\geq 1}$ be a sequence defined by
$q_n=q_{F_0}^{\varepsilon_0}\ldots q_{F_N}^{\varepsilon_N}$, where
$\displaystyle n=\sum_{i=0}^N\varepsilon_iF_i$ with
$\varepsilon_i\in\{0,1\}$, $\varepsilon_{i+1}\varepsilon_{i}\neq
11$, $q_{F_0}=\frac{\lambda-(1-p_1)}{p_1}$,
$q_{F_1}=\frac{1}{p_2}q_{F_0}^2-\left(\frac{1}{p_2}-1\right)$ and
$q_{F_n}=\frac{1}{r_n}q_{F_{n-1}}q_{F_{n-2}}-\left(\frac{1}{r_n}-1\right)$,
where $r_n=p_{\left[\frac{n+1}{2}\right]+1}$, for all $n\geq 1$.
\label{defi431}
\end{Definition}

\begin{Theorem} Acting in $l^\infty(\mathbb{N})$, we have
$\sigma_{pt}(S)=\{\lambda\in\mathbb{C}:(q_n(\lambda))_{n\geq 0}
\textrm{ is bounded}\}$. \label{teorema431}
\end{Theorem}

\begin{Remark}
In particular, $\sigma_{pt}(S)\subset E:=\{\lambda\in\mathbb{C}:
(q_{F_n}(\lambda))_{n\geq 0} \textrm{ is
bounded}\}$ and $E=\{z\in\mathbb{C}:
(\psi_n (z,z))_{n\geq 0} \textrm{ is
bounded}\}$, where $\psi_n=g_n\circ
g_{n-1}\circ\ldots\circ g_0$ and $g_n:\mathbb{C}^2\longrightarrow\mathbb{C}^2$ are maps defined by
$g_0(x,y)=\left(\frac{x-(1-p_1)}{p_1},\frac{y-(1-p_1)}{p_1}\right)$ and $g_n(x,y)=\left(\frac{1}{r_n}xy-\left(\frac{1}{r_n}-1\right),x\right)$ for all $n\geq 1$.\label{obslazara}
\end{Remark}

\np \textbf{Remark} The proof of the theorem \ref{teorema431} is similar to the demonstration of the same theorem proved in \cite{ms}, where $p_i=p\in]0,1]$ is fixed,
for all $i\geq 1$. To help the reader, the proofs will be displayed.

Before to proof the theorem \ref{teorema431}, it will be necessary
the followings results:

\begin{Lemma} Let $\lambda\in\mathbb{C}$ and $v\in l^\infty(\mathbb{N})$
such that $Sv=\lambda v$. Then for each $k\geq 1$, there exists a
complex number $\beta_k=\beta_k(\overline{p},\lambda)$, where
$\overline{p}=(p_i)_{i\geq 1}$, such that $v_k=\beta_kv_0$.
Furthermore, if $m,n\in\mathbb{N}$, with $n>1$, satisfies
$0<m<F_{n-1}$ then $\beta_{m+F_n}=\beta_{F_n}\beta_m$. In
particular, if $n=\sum_{i=0}^Nb_iF_i$, where
$b_ib_{i-1}\displaystyle<_{lex}11$, for all $i\in\{1,\ldots,N\}$
then $\beta_n=\beta_{F_0}^{b_0}\ldots \beta_{F_n}^{b_N}$.
\label{lema432}
\end{Lemma}
\begin{proof} Let $\lambda$ be an eigenvalue of $S$ associated to the
eigenvector $v=(v_i)_{i\geq 0}\in l^\infty (\mathbb{N})$. Since the
transition probability from any nonnegative integer $i$ to any
integer $i+k$, with $k\geq 2$, is $p_{i,i+k}=0$ (see proposition
(\ref{prop431})), the operator $S$ satisfies $S_{i,i+k}=0$, for all
$i,k\in\mathbb{N}$, with $k\geq 2$. Hence, for every integer $k\geq
1$, we have
\begin{equation}
\displaystyle\sum_{i=0}^kS_{k-1,i}v_i=\lambda v_{k-1}. \label{42}
\end{equation}
\np Then it is possible to prove by induction on $k$ that for all
integer $k\geq 1$, there exists a complex number
$\beta_k=\beta_k(\overline{p},\lambda)\in\mathbb{C}$ such that
$v_k=\beta_k v_0$.

Put $k=F_n+j$, $j\in\{1,\ldots ,m\}$, where $0<m<F_{n-1}$. From
proposition \ref{prop433}, we have that $S_{k-1,i}=0$, for all
$i\in\{0,\ldots,F_n-1\}$. Hence, for all $j=1,\ldots,m$, we have
that

\begin{center}
$\displaystyle\sum_{i=F_n}^{F_n+j}S_{F_n+j-1,i}v_i=\lambda
v_{F_n+j-1}$.
\end{center}

From proposition \ref{prop432} ($S_{i,j}=S_{i-F_n,j-F_n}$) we deduce that, for all $j=1,\ldots,m$,
\begin{equation}
\lambda
v_{F_n+j-1}=\displaystyle\sum_{i=F_n}^{F_n+j}S_{F_n+j-1,i}v_i=
\sum_{i=F_n}^{F_n+j}S_{j-1,i-F_n}v_i=
\sum_{l=0}^jS_{j-1,l}v_{l+F_n}. \label{45}
\end{equation}

Let $(w_i)_{0\leq i\leq j}$ defined by $w_i=v_{F_n+i}$. Hence,
$\sum_{i=0}^jS_{j-1,i}w_i=\lambda w_{j-1}$, for $j\in\{1,\ldots,m\}$.

Therefore, $w_m=\beta_m w_0$, i.e. $v_{m+F_n}=\beta_mv_{F_n}$.
Since $v_{F_n}=\beta_{F_n}v_0$ and $v_{F_n+m}=\beta_{F_n+m}v_0$, it
follows that $\beta_{F_n+m}=\beta_m\beta_{F_n}$.
\end{proof}

\np \textbf{Proof of the theorem \ref{teorema431}:} Let $\lambda$ be
an eigenvalue of $S$ associated to the eigenvector $v=(v_i)_{i\geq
0}\in l^\infty (\mathbb{N})$. Hence, from lemma \ref{lema432} we
have that for all nonnegative integer $k\geq 1$, there exists a complex
number $\beta_k=\beta_k(\overline{p},\lambda)\in\mathbb{C}$ such
that
\begin{equation}
v_k=\beta_k v_0.  \label{43}
\end{equation}
\np \textbf{Claim:} $\beta_k=q_k$, for all $k\geq 1$.

Put $k=F_{2n-1}$ on relation (\ref{42}). Using the fact that
$F_{2n-1}-1=\displaystyle\underbrace{1010\ldots
10}_{2n-2}1=(10)^{n-1}1$ and the item $(iv)$ from proposition
\ref{prop431}, we obtain, for all nonnegative integer $n\geq 1$,

\begin{center}
$(1-p_1)v_{F_{2n-1}-1}+\displaystyle\sum_{i=1}^np_1p_2\ldots
p_i(1-p_{i+1})v_{F_{2n-1}-F_{2i-1}}+p_1p_2\ldots
p_{n+1}v_{F_{2n-1}}=\lambda v_{F_{2n-1}-1}$,
\end{center}
\np i.e.
\begin{equation}
v_{F_{2n-1}}=\frac{-1}{p_1\ldots
p_{n+1}}\left((1-p_1-\lambda)v_{F_{2n-1}-1}+\displaystyle\sum_{i=1}^n
p_1\ldots p_i(1-p_{i+1})v_{F_{2n-1}-F_{2i-1}}\right). \label{48c}
\end{equation}
From relation (\ref{48c}), changing $n$ by $n+1$, we obtain
\begin{equation}
v_{F_{2n+1}}=\frac{-1}{p_1\ldots
p_{n+2}}\left((1-p_1-\lambda)v_{F_{2n+1}-1}+\displaystyle\sum_{i=1}^{n+1}
p_1\ldots p_i(1-p_{i+1})v_{F_{2n+1}-F_{2i-1}}\right). \label{48}
\end{equation}
Hence, we have
$v_{F_{2n+1}}=\left(1-\frac{1}{p_{n+2}}\right)v_0+\frac{1}{p_{n+2}}A_{F_{2n+1}}$,
where $A_{F_{2n+1}}$ is given by:

\begin{center}
$\frac{-1}{p_1\ldots
p_{n+1}}\left((1-p_1-\lambda)v_{F_{2n-1}+F_{2n}-1}+
\displaystyle\sum_{i=1}^n p_1\ldots
p_i(1-p_{i+1})v_{F_{2n-1}+F_{2n}-F_{2i-1}}\right)$.
\end{center}

From relation (\ref{43}) and lemma \ref{lema432}, we have

\begin{center}
$v_{F_{2n-1}+F_{2n}-1}=\beta_{F_{2n-1}+F_{2n}-1}v_0=\beta_{F_{2n-1}-1}\beta_{F_{2n}}v_0=\beta_{F_{2n}}v_{F_{2n-1}-1}$.
\end{center}

We also obtain $v_{F_{2n-1}+F_{2n}-F_{2i-1}}= \beta_{F_{2n}}v_{F_{2n-1}-F_{2i-1}}$,
for all $i\in\{1,\ldots ,n\}$.

Hence, from relations (\ref{48c}) and (\ref{48}) we obtain
$A_{F_{2n+1}}=\beta_{F_{2n}}v_{F_{2n-1}}=\beta_{F_{2n}}\beta_{F_{2n-1}}v_0$.

Therefore, we conclude that $\beta_{F_{2n+1}}=\frac{1}{p_{n+2}}\beta_{F_{2n}}\beta_{F_{2n-1}}-\left(\frac{1}{p_{n+2}}-1\right)$,
for all $n\geq 1$.

The case $k=F_{2n}$ can be done by the same way.

On the other hand, it's easy to check that $v_{F_0}=q_{F_0}v_0$ and
$v_{F_1}=q_{F_1}v_0$, i.e.

\begin{center}
$\beta_{F_0}=-\frac{1-\lambda-p_1}{p_1}=q_{F_0}$ \hspace{.3cm} and
\hspace{.3cm}
$\beta_{F_1}=\frac{1}{p_2}q_{F_0}^2-\left(\frac{1}{p_2}-1\right)=q_{F_1}
$.
\end{center}

Therefore, $\beta_n=q_n$, for all $n\geq 1$.

Hence, we have
$\sigma_{pt}(S)=\{\lambda\in\mathbb{C}:(q_n(\lambda))_{n\in\mathbb{N}}
\textrm{ is bounded}\}$ and we are done.

\rightline{$\Box$}

\begin{Conjecture}
There exists $0<d<1$ such that if $p_i>d$, for all $i\geq 1$ then
$\sigma_{pt}(S)=E$. \label{conjec}
\end{Conjecture}

\begin{Remark}
In \cite{cm}, we prove a particular
case of conjecture \ref{conjec}. Particulary we prove that if
$p_i=p$, for all $i\geq 1$, and $\frac{-1+\sqrt{5}}{2}<p<1$ then
$E\cap \mathbb{R}=\sigma_{pt}(S)\cap\mathbb{R}$.
\end{Remark}

\begin{Theorem}
In $l^\infty(\mathbb{N})$, the set $E$ is contained in $\sigma_a(S)$. In particular $E\subset\sigma(S)$.
\end{Theorem}
\begin{proof} Let $\lambda\in E$ and suppose that
$\lambda\notin\sigma_{pt}(S)$. We will prove that
$\lambda\in\sigma_a(S)$. In fact, for each $k\geq 2$, consider

\begin{center}
$w^{(k)}=(w^{k}_0,w^{k}_1,w^{k}_2,\ldots,w^{k}_k,w^{k}_{k+1},w^{k}_{k+2},\ldots)^t=
(1,q_1(\lambda),q_2(\lambda),\ldots,q_k(\lambda),0,0,\ldots)^t$,
\end{center}

\np where $(q_n(\lambda))_{n \geq 1}=(q_n)_{n \geq 1}$ is the
sequence defined before. Define $u^{(k)}:=\frac{w^{(k)}}{\lVert
w^{(k)}\rVert_{\infty}}$.

\np \textbf{Claim:} $\displaystyle\lim_{n\to+\infty}\lVert(S-\lambda
I)u^{(F_n)}\rVert_\infty=0$.

\np In fact, for all $i\in\{0,\ldots,k-1\}$, we have $((S-\lambda
I)u^{(k)})_i=0$ and $u_i=0$, for all $i>k$. Hence, note that
\begin{center}
$\|(S-\lambda I)u^{(k)}\|_\infty= \displaystyle\sup_{i\geq
0}\left|\sum_{j=0}^{+\infty}(S-\lambda I)_{ij}u_j^{(k)}\right|=
\displaystyle\sup_{i\geq
k}\left\{\frac{\left|\sum_{j=0}^k(S-\lambda
I)_{ij}w_j^{(k)}\right|}{\| w^{(k)}\|_\infty}\right\}$.
\end{center}

If $k=F_n$ then, for $i\geq k=F_n$, we have:

\np $a)$ If $i=F_n$ then $S_{i,j}=0$, for all $j\in\{0,\ldots,
F_n-1\}$, and $S_{F_n,F_n}=1-p_1$. Therefore,
$\left|\sum_{j=0}^{F_n}(S-\lambda
I)_{ij}w_j^{(F_n)}\right|=|1-p_1-\lambda||q_{F_n}|$.

\np $b)$ If $F_n<i<F_{n+1}-1$ then $S_{i,j}=0$, for all
$j\in\{0,\ldots, F_n-1\}$, and $S_{i,j}\leq p_1$, for $j=F_n$.
Therefore, $\left|\sum_{j=0}^{F_n}(S-\lambda I)_{ij}w_j^{(F_n)}\right|\leq
p_1|q_{F_n}|$.

\np $c)$ If $i=F_{n+1}-1$ then $S_{i,j}=0$, for all
$j\in\{1,\ldots, F_n-1\}$, and $S_{i,j}\leq p_1$, for $j=0,F_n$. Therefore,
$\left|\sum_{j=0}^{F_n}(S-\lambda I)_{ij}w_j^{(F_n)}\right|\leq
p_1+p_1|q_{F_n}|$.

\np $d)$ If $i\geq F_{n+1}$ then $S_{i,j}=0$, for all
$j\in\{1,\ldots, F_n\}$, and $S_{i,j}\leq p_1$, for $j=0$. Therefore,
$\left|\sum_{j=0}^{F_n}(S-\lambda I)_{ij}w_j^{(F_n)}\right|\leq
p_1$.

Hence, from $a)$, $b)$, $c)$ and $d)$ it follows that
\begin{equation}
\|(S-\lambda I)u^{(F_n)}\|_\infty\leq
\frac{|1-p_1-\lambda||q_{F_n}|+p_1|q_{F_n}|+p_1}{\|
w^{(F_n)}\|_\infty}. \label{desi}
\end{equation}
Since $\lambda\in E$ and $\lambda\notin\sigma_{pt}(S)$ if follows
that $(q_{F_n})_{n\geq 0}$ is a bounded sequence and $(q_n)_{n\geq
0}$ is not. Therefore, we have $\lim_{n\to+\infty}\lVert
w^{(F_n)}\rVert_\infty=+\infty$, which
implies from relation \ref{desi} that
$\displaystyle\lim_{n\to+\infty}\lVert(S-\lambda
I)u^{(F_n)}\rVert_\infty=0$. Therefore,
$\lambda\in\sigma_a(S)\subset\sigma(S)$.
\end{proof}

\begin{Conjecture}
$E=\sigma_a(S)$.
\end{Conjecture}

\subsection{Generalization}

Let $d> 2$ a integer number and let us consider the sequence $(F_n)_{n\geq 0}$ given by $F_n=a_1 F_{n-1}+\ldots+a_d F_{n-d}$, for all $n\geq d$, with initial conditions $F_0=1$ and $F_n=a_1 F_{n-1}+\ldots +a_n F_0+1$, for all $n\in\{1,\ldots, d-1\}$, where $a_i$ are nonnegative integers, for $i\in\{1,\ldots, d\}$, satisfying $a_1\geq a_2\geq\ldots\geq a_d\geq 1$.

By using the greedy algorithm we can write every nonnegative integer $N$, in a unique way, as $N=\sum_{i=0}^{k(N)}\varepsilon_i(N)F_i$ where the digits $\varepsilon_j(N)$ satisfy the relation
$\varepsilon_i\varepsilon_{i-1}\ldots\varepsilon_{i-d+1}<_{lex}a_1a_2\ldots a_d$, for all $i\geq d-1$.

It is known that the addition of $1$ in base $(F_n)_{n\geq 0}$ is given by a finite
transducer. By using the same construction as was done in Fibonacci base, we can define the stochastic adding machine associated to the sequences $(F_n)_{n\geq 0}$ and $\overline{p}=(p_i)_{i\geq 1}$ and we can also prove that
$\sigma_{pt}(S_{\overline{p}})=\{\lambda\in\mathbb{C}:(q_n(\lambda))_{n\geq
0}\textrm{ is bounded}\}$, where $q_{F_i}(\lambda)$ are polynomials fixed in $\lambda$ for all $i\in\{0, \ldots,d-1\}$,

\begin{center}
$q_{F_{dn+i}}(\lambda)=\frac{1}{p_{n+1+i}}q_{F_{dn+i-1}}^{a_1}(\lambda)q_{F_{dn+i-2}}^{a_2}(\lambda)\ldots
q_{F_{dn+i-d}}^{a_d}(\lambda)-\left(\frac{1}{p_{n+1+i}}-1\right)$
\end{center}

\np for all $n\geq 1$ and for each $i\in\{0,1,\ldots,d-1\}$, and $q_N(\lambda)=q_{F_0}^{\varepsilon_0}(\lambda)\ldots
q_{F_n}^{\varepsilon_{n}}(\lambda)$, where
$N=\sum_{i=0}^{n}\varepsilon_iF_i$.

In particular, if $a_1=a_2=\ldots=a_d=1$ then $q_{F_0}(\lambda)=\frac{\lambda-(1-p_1)}{p_1}$ and $q_{F_i}(\lambda)=\frac{1}{p_{i+1}}(q_{F_{i-1}}(\lambda))^2-\left(\frac{1}{p_{i+1}}-1\right)$, for all $i\in\{1,\ldots,d-1\}$.

\section{Topological properties of the stochastic Fibonacci adding machine}

\begin{Theorem} Suppose there exists $\delta>0$ such that $p_i>\delta$,
for all $i\geq 0$. Then $E$ is compact and $\mathbb{C}\setminus E$
is connected. \label{teoremacomp}
\end{Theorem}

In definition \ref{defi431}, we have that for each $n\in\mathbb{N}$, $q_{F_n}:\mathbb{C}\longrightarrow\mathbb{C}$ are maps defined by
$q_{F_0}(z)=\frac{z-(1-p_1)}{p_1}$,
$q_{F_1}(z)=\frac{1}{p_2}(q_{F_0}(z))^2-\left(\frac{1}{p_2}-1\right)$ and
$q_{F_n}(z)=\frac{1}{r_n}q_{F_{n-1}}(z)q_{F_{n-2}}(z)-\left(\frac{1}{r_n}-1\right)$,
where $r_n=p_{\left[\frac{n+1}{2}\right]+1}$, for all $n\geq 1$ and the set $E$ is defined by $E:=\{z\in \mathbb{C}:(q_{F_n}(z))_{n\geq 0} \textrm{ is bounded}\}$.
 
Before to proof the theorem \ref{teoremacomp}, it will be necessary
the followings results:

\begin{Lemma} If there exists a nonnegative integer $k$ such that $|q_{F_n}(z)|>1$
for $n\in\{k,k+1\}$ then the sequence $(q_{F_n}(z))_{n\geq 0}$
is not bounded. \label{lema441}
\end{Lemma}
\begin{proof} Since $|q_{F_n}(z)|>1$ for $n\in\{k,k+1\}$ then there exists
a real number $A>1$ such that $|q_{F_n}(z)|>A$ for $n\in \{k,k+1\}$.
Hence, we have

\begin{center}
$|q_{F_{k+2}}(z)|\geq\frac{1}{r_{k+2}}|q_{F_{k+1}}(z)q_{F_k}(z)|-|\frac{1}{r_{k+2}}-|>
\frac{A^2}{r_{k+2}}-\frac{1}{r_{k+2}}+1$, i.e. $|f_{k+2}(z)|>A^2$
\end{center}

\np and likewise we have $|q_{F_{k+3}}(z)|> A^3$ and $|q_{F_{k+4}}(z)|>A^5$.

Continuing in this way, we conclude that $|q_{F_{k+l}}(z)|>A^{F_{l-1}}$, for all $l\in\mathbb{N}$, $l\geq 2$.

Since $A>1$, it follows that $(q_{F_n}(z))_{n\geq 0}$ is not bounded.
\end{proof}

\begin{Proposition} Suppose there exists $\delta>0$ such that $p_i>\delta$,
for all $i\geq 1$. Then there exists $R>1$ such that,
if $|q_{F_k}(z)|>R$ for some $k\in\mathbb{N}$ then
$(q_{F_n}(z))_{n\geq 0}$ is not bounded. \label{prop441}
\end{Proposition}
\begin{proof} Let $R>1$ be a real number such that $R>\frac{2}{\delta}-1$.
Hence, we have that
\begin{equation}
\frac{2}{p_2}-1<R<R\delta\left(R+1-\frac{1}{\delta}\right).
\label{R}
\end{equation}
If $|q_{F_0}(z)|>R$ then $|q_{F_1}(z)|\geq
\frac{1}{p_2}|q_{F_0}(z)|^2-\left|\frac{1}{p_2}-1\right|
>\frac{1}{p_2}R^2-\left(\frac{1}{p_2}-1\right)$, i.e. $|q_{F_1}(z)|>R^2$. Therefore, it follows from lemma \ref*{lema441} that the sequence $(q_{F_n}(z))_{n\geq 0}$ is not bounded.

If $|q_{F_1}(z)|>R$ then $R<|q_{F_1}(z)|\leq\frac{1}{p_2}|(q_{F_0}(z))^2|+\left|\frac{1}{p_2}
-1\right|=\frac{1}{p_2}|q_{F_0}(z)|^2+\frac{1}{p_2}-1$ and it implies
$|q_{F_0}(z)|^2>p_2\left(R+1-\frac{1}{p_2}\right)>\delta\left(R+1-\frac{1}{\delta}\right)>1$, where the last inequality follows from relation (\ref{R}). Hence, $|q_{F_0}(z)|>1$ and from lemma
\ref*{lema441} we have that the sequence $(q_{F_n}(z))_{n\geq 0}$ is not
bounded.

By induction on $k$, suppose that for all $i\in\{0,1,\ldots,k-1\}$
we have: if $|q_{F_i}(z)|>R$ then $(q_{F_n}(z))_{n\geq 0}$ is not bounded.

Suppose that $|q_{F_k}(z)|>R$. If $|q_{F_{k+1}}(z)|>1$, it follows from lemma \ref*{lema441}
that $(q_{F_n}(z))_{n\geq 0}$ is not bounded. Therefore, suppose $|q_{F_{k+1}}(z)|\leq 1$. Then we have

\begin{center}
$1\geq
|q_{F_{k+1}}(z)|\geq
\frac{1}{r_{k+1}}|q_{F_k}(z)||q_{F_{k-1}}(z)|-\left|\frac{1}{r_{k+1}}-1
\right|> \frac{R}{r_{k+1}}|q_{F_{k-1}}(z)|-\left(\frac{1}{r_{k+1}}-1\right)$,
\end{center}

\np i.e. $\frac{1}{R}>|f_{k-1}(z)|$. On the other hand, we have that

\begin{center}
$R<|q_{F_k}(z)|\leq
\frac{1}{r_k}|q_{F_{k-1}}(z)||q_{F_{k-2}}(z)|+\left|\frac{1}{r_k}-1\right|< \frac{1}{r_k}\frac{1}{R} |q_{F_{k-2}}(z)|+\frac{1}{r_k}-1$,
\end{center}

\np i.e. $|q_{F_{k-2}}(z)|>Rr_k\left(
R+1-\frac{1}{r_k}\right)>R\delta\left( R+1-\frac{1}{\delta}\right)>R$. Then by induction hypothesis we have that the sequence $(q_{F_n}(z))_{n\geq 0}$ is not bounded.
\end{proof}

\begin{Proposition}
Suppose there exists $\delta>0$ such that $p_i>\delta$, for all
$i\geq 1$. Then there exists a real number $R>1$ such that
$E=\displaystyle\bigcap_{n=0}^{+\infty}
q_{F_n}^{-1}\overline{D(0,R)}$, where $\overline{D(0,R)}$ is the closed
disk of centre $0$ and radius $R$. \label{prop442}
\end{Proposition}
\begin{proof} It is a direct consequence of proposition \ref{prop441}.
\end{proof}

\begin{Remark}
From proof of proposition \ref{prop441}, we also have $|q_{F_k}(z)|\geq
R$ for some nonnegative integer $k$ implies $(q_{F_n}(z))_{n\geq 0}$ is
not bounded. Hence, $E=\bigcap_{n=0}^{+\infty} q_{F_n}^{-1}D(0,R)$.
\label{obsg}
\end{Remark}

\np \textbf{Proof of theorem \ref{teoremacomp}:} It follows directly from proposition \ref{prop442} that $E$ is
compact.

From proposition \ref{prop441}, there exists $R>0$ such that $\mathbb{C}\setminus
E=\bigcup_{n=0}^{+\infty}\mathbb{C}\setminus
q_{F_n}^{-1}\overline{D(0,R)}$. Hence, since $\mathbb{C}\setminus\overline{D(0,R)}$ is
connected, it follows from \emph{maximum modulus principle} that for
each holomorphic map $q_{F_n}$, $\mathbb{C}\setminus
q_{F_n}^{-1}\overline{D(0,R)}$ is connected for all $n\geq 0$. On the
other hand, since $\mathbb{C}\setminus q_{F_n}^{-1}\overline{D(0,R)}$
contains a neighbourhood of infinity for all $n\geq 0$, we deduce
that

\begin{center}
$\mathbb{C}\setminus E=\bigcup_{n=0}^{+\infty}\mathbb{C}\setminus
q_{F_n}^{-1}\overline{D(0,R)}$ is connected.
\end{center}

\rightline{$\Box$}

\begin{Conjecture}
$E$ is a compact set and $\mathbb{C}\setminus E$ is a connected set,
even supposing $\displaystyle\liminf_{i\to+\infty} p_i=0$.
\end{Conjecture}

\begin{Proposition}
Suppose there exists $\delta>0$ such that $p_i>\delta$, for all
$i\geq 1$. Then there exists $R>1$ such that
$E=\bigcap_{n=0}^{+\infty}q_{F_n}^{-1}D(0,R)$
and $q_{F_{n+1}}^{-1}D(0,R)\subset q_{F_n}^{-1}D(0,R)$, for all $n\geq 0$.
\label{prop443}
\end{Proposition}
\begin{proof} Let $R$ be a real number such that $R>\frac{2}{\delta}-1$.
From remark \ref{obsg}, we have that $E= \bigcap_{n=0}^{+\infty}q_{F_n}^{-1}D(0,R)$.

We will prove by induction on $k$ that $q_{F_{k+1}}^{-1}D(0,R)\subset
q_{F_k}^{-1}D(0,R)$, for all $k\geq 0$.

In fact, suppose that $q_{F_{k+1}}^{-1}D(0,R)\subset q^{-1}_{F_k}D(0,R)$,
for all $k\in\{0,\ldots,n-1\}$, with $n\geq 4$. Let $z\in
q^{-1}_{F_{n+1}}D(0,R)$ and suppose that $q_{F_n}(z)\geq R$.

Let $A=1-\frac{1}{\delta}$. Hence, $\left|1-\frac{1}{r_n}\right|<|A|$
for all $n\geq 1$ and

\begin{center}
$R>|q_{F_{n+1}}(z)|\geq
\frac{1}{r_{n+1}}R|q_{F_{n-1}}(z)|-\left|\frac{1}{r_{n+1}}-1\right|$, i.e.
$|q_{F_{n-1}}(z)|<\frac{r_{n+1}(R+|A|)}{R}=\mathcal{O}(1)$.
\end{center}

Therefore, by induction hypothesis we have that
\begin{equation}
|q_{F_k}(z)|<R \textrm{ for all } k\in \{1,\ldots,n-2\}. \label{419}
\end{equation}
On the other hand

\begin{center}
$|q_{F_n}(z)|\leq\frac{1}{r_n}|q_{F_{n-1}}(z)||q_{F_{n-2}}(z)|+\left|\frac{1}{r_n}-1\right|
<\frac{1}{r_n}\mathcal{O}(1)|q_{F_{n-2}}(z)|+|A|$, i.e.
$|q_{F_{n-2}}(z)|> r_n\frac{|q_{F_n}(z)|-|A|}{\mathcal{O}(1)}>\delta\dfrac{R-|A|}{\mathcal{O}(1)}
= \mathcal{O}(R)$.
\end{center}

\np Thus, continuing this way we have that $|q_{F_{n-3}}(z)|<\mathcal{O}\left(\frac{1}{R}\right)$ and
$|q_{F_{n-4}}(z)|> \mathcal{O}(R^2)$.

\np Choosing $R$ large enough, we have $|q_{F_{n-4}}(z)|>R$ and this contradicts the relation (\ref{419}).
Therefore, $z\in q^{-1}_{F_n}(D(0,R))$.

Now, in order to finish the proof of this theorem, we will prove
that $q_{F_{k+1}}^{-1}D(0,R)\subset q_{F_k}^{-1}D(0,R)$, for $k=0,1,2,3$.

\np \textbf{Case k=0:} Let $z\in q_{F_1}^{-1}D(0,R)$. Then
$\frac{1}{p_2}|q_{F_0}(z)|^2-\left(\frac{1}{p_2}-1\right)<R$. Therefore,

\begin{center}
$\frac{1}{p_2}(|q_{F_0}(z)|^2-1)<R-1$ and $|q_{F_0}(z)|<\sqrt{R}<R$, i.e. $z\in q^{-1}_{F_0} D(0,R)$.
\end{center}

\np \textbf{Case k=1:} Let $z\in q_{F_2}^{-1}D(0,R)$. Then $\frac{1}{r_2}|q_{F_1}(z)||q_{F_0}(z)|-\left(\frac{1}{r_2}-1\right)<R$. Therefore,
$|q_{F_1}(z)||q_{F_0}(z)|< r_2R+1-r_2 < R$. Thus, we have that $|q_{F_1}(z)|<(r_2R+1-r_2)^{\frac{3}{4}}$ or
$|q_{F_0}(z)|<(r_2R+1-r_2)^{\frac{1}{4}}$.

If $|q_{F_1}(z)|<(r_2R+1-r_2)^{\frac{3}{4}}<R$ then $z\in q_{F_1}^{-1}D(0,R)$.

If $|q_{F_0}(z)|<(r_2R+1-r_2)^{\frac{1}{4}}$ then

\begin{center}
$|q_{F_1}(z)|\leq \frac{1}{r_2}|q_{F_0}(z)|^2+\left|\frac{1}{r_2}-1\right|\leq
\frac{1}{r_2}(r_2R+1-r_2)^{\frac{1}{2}}+\left|\frac{1}{r_2}-1\right|$.
\end{center}

Thus, choosing $R$ large enough such that $\frac{1}{r_2}(r_2R+1-r_2)^{\frac{1}{2}}+\left|\frac{1}{r_2}-1\right|<R$, we have
$z\in q_{F_1}^{-1}D(0,R)$.

\np \textbf{Case k=2:} Let $z\in q_{F_3}^{-1}D(0,R)$. Then $\frac{1}{r_3}|q_{F_2}(z)||q_{F_1}(z)|-\left(\frac{1}{r_3}-1\right)<R$. Therefore,
$|q_{F_2}(z)||q_{F_1}(z)|< r_3R+1-r_3 \leq R$. Thus, we have that $|q_{F_2}(z)|<(r_3R+1-r_3)^{\frac{1}{2}}$ or
$|q_{F_1}(z)|<(r_3R+1-r_3)^{\frac{1}{2}}$.

If $|q_{F_2}(z)|<(r_3R+1-r_3)^{\frac{1}{2}}\leq R$ then $z\in q_{F_2}^{-1}D(0,R)$.

If $|q_{F_1}(z)|<(r_3R+1-r_3)^{\frac{1}{2}}$ then $\frac{1}{p_2}|q_{F_0}(z)|^2-\left(\frac{1}{p_2}-1\right)<(r_3R+1-r_3)^{\frac{1}{2}}$, i.e.
$|q_{F_0}(z)|<\sqrt{p_2}\left((r_3R+1-r_3)^{\frac{1}{2}}+\left(\frac{1}{p_2}-1\right)\right)^\frac{1}{2}$. Hence,
$|q_{F_2}(z)|\leq \frac{1}{p_2}|q_{F_1}(z)||q_{F_0}(z)|+ \left(\frac{1}{p_2}-1\right)$, i.e.
$|q_{F_2}(z)| < \frac{1}{p_2}(r_3R+1-r_3)^{\frac{1}{2}}\sqrt{p_2}\left((r_3R+1-r_3)^{\frac{1}{2}}+  \left(\frac{1}{p_2}-1\right)\right)^\frac{1}{2} + \left(\frac{1}{p_2}-1\right)=\mathcal{O}(R^{\frac{3}{4}}).$

Thus, choosing $R$ large enough, we have $|q_{F_2}(z)|<R$, i.e.
$z\in q_{F_2}^{-1}D(0,R)$.

\np \textbf{Case k=3} This case can be done by the same way.
\end{proof}

Let $h:\mathbb{C}\longrightarrow\mathbb{C}$ be a non-null polynomial
and define the set $\mathcal{F}_h:=\{z\in\mathbb{C}:(g_n\circ\ldots\circ g_2(h(z),z))_{n\geq 2}$
is bounded$\}$, where
$g_n:\mathbb{C}^2\longrightarrow\mathbb{C}^2$ is defined in the remark \ref{obslazara}.

\begin{Remark}
If $h(z)=\frac{1}{p_2}z^2-\left(\frac{1}{p_2}-1\right)$ then $E$
and $\mathcal{F}_h$ are isomorphic by $l$, where
$l:\mathbb{C}\longrightarrow\mathbb{C}$ is defined by
$l(\lambda)=\frac{1}{p_1}\lambda-\left(\frac{1}{p_1}-1\right)$. \label{obs4422}
\end{Remark}

\begin{Claim}
The above results are true, if we consider the set $\mathcal{F}_h$,
for all non-null polynomial $h$. In particular, if we suppose there
exists $\delta>0$ such that $p_i>\delta$, for all $i\geq 1$ then we
have that $\mathbb{C}\setminus\mathcal{F}_h$ is a connected set and
there exists $R>1$ such that $\mathcal{F}_h=\bigcap_{n=0}^{\infty}
\varphi_n^{-1}D(0,R)$, where $\varphi_0(z)=z$, $\varphi_1(z)=h(z)$
and $\varphi_n(z)=
\frac{1}{r_n}\varphi_{n-1}(z)\varphi_{n-2}(z)-\left(\frac{1}{r_n}-1\right)$,
for all $n\geq 2$. Furthermore we have that
$\varphi^{-1}_{n+1}D(0,R)\subset\varphi^{-1}_nD(0,R)$, for all
$n\in\mathbb{N}$. \label{obs442}.
\end{Claim}

\begin{Theorem}
Let $h(z)=a_2z^2+a_3z^3+\ldots+a_nz^n$, with $n\geq 2$ and
$a_2,\ldots,a_n\in\mathbb{C}$, and suppose that
$\displaystyle\liminf_{i\to+\infty} p_i>0$. Then we have the
followings results:

\np $a)$ If $p_3<\frac{1}{2}$ and $p_2=1$ or $0<p_2<1-p_3$
then $\mathcal{F}_h$ is a non-connected set.

\np $b)$ If $p_i=1$ for all $i\in\{2,3,\ldots,k\}$ and
$p_{k+1}<\frac{1}{2}$, for some $k\geq 3$ then $\mathcal{F}_h$ is a
non-connected set. \label{teorema442}
\end{Theorem}

Before to prove the theorem \ref{teorema442}, it will be necessary
the following lemma which is a particular case of the
Riemann-Hurwitz formula (see the theorem $7.2$ of the page $70$ in \cite{m} and the theorem $1.1.4$ and the lemma $1.1.5$ of the page $10$ in \cite{mntu}).

\begin{Lemma}
Let $\phi:\mathbb{C}\longrightarrow\mathbb{C}$ be a polynomial
function and $R>0$. Then

\begin{center}
$\phi^{-1}D(0,R)$ is connected if only if
$\{z\in\mathbb{C}:\phi'(z)=0\}\subset\phi^{-1}D(0,R)$.
\end{center}
\label{lema443}
\end{Lemma}

\np \textbf{Proof of theorem \ref{teorema442}:}

\np $a)$ If $p_2=1$ and $p_3<\frac{1}{2}$ then $\varphi_0(0)=\varphi_1(0)=\varphi_2(0)=0$ and
$|\varphi_3(0)|=|\varphi_4(0)|=\frac{1}{p_3}-1>1$,

\np which implies from lemma \ref{lema441} that
$(\varphi_n(0))_{n\geq 0}$ is not bounded.

If $p_3<\frac{1}{2}$ and $0<p_2<1-p_3$ then
$\varphi_0(0)=\varphi_1(0)=0$, $\varphi_2(0)=1-\frac{1}{p_2}$,
$|\varphi_3(0)|=\frac{1}{p_3}-1>1$. Furthermore, $p_2<1-p_3$ implies 
$\varphi_4(0)=\frac{1}{p_3}\left(1-\frac{1}{p_3}\right)\left(1-\frac{1}{p_2}\right)-\left(\frac{1}{p_3}-1\right)>1$.

\np Thus, since $|\varphi_3(0)|>1$ and $|\varphi_4(0)|>1$, it
follows from lemma \ref{lema441} that $(\varphi_n(0))_n$ is not
bounded.

Therefore, in both cases we have that $0\notin\mathcal{F}_h$.

Hence, there exists a integer number $n\geq 1$ such that
$0\notin\varphi^{-1}_kD(0,R)$, for all $k\geq n$, where $R$ is the
real number defined at remark \ref{obs442}.

Since $\varphi'_n(0)=0$, for all $n\geq 1$, we deduce from lemma
\ref{lema443} that $\varphi^{-1}_kD(0,R)$ is not connected for all
$k\geq n$. Since $\mathcal{F}_h=\bigcap_{n=0}^{\infty}
\varphi_n^{-1}D(0,R)$ and
$\varphi^{-1}_{n+1}D(0,R)\subset\varphi^{-1}_nD(0,R)$, for all
$n\in\mathbb{N}$, it follows that $\mathcal{F}_h$ is not connected.

\np $b)$ Let $k\geq 3$ and suppose that $p_i=1$ for all $i\in\{2,3,\ldots,k\}$ and
$p_{k+1}<\frac{1}{2}$. Hence, we have that
$\varphi_0(0)=\varphi_1(0)=\ldots=\varphi_{2k-2}(0)=0$ and
$|\varphi_{2k-1}(0)|=|\varphi_{2k}(0)|=\frac{1}{p_{k+1}}-1>1$. Thus,
in the same way that was done in item $a)$, it follows that
$\mathcal{F}_h$ is not connected.

\rightline{$\Box$}

\begin{Corollary}
Suppose that $\displaystyle\liminf_{i\to+\infty} p_i>0$. If $p_i=1$
for all $i\in\{2,3,\ldots,k\}$ and $p_{k+1}<\frac{1}{2}$, for some
$k\geq 2$ then $E$ is a non-connected set.
\end{Corollary}
\begin{proof} Let $h(z)=\frac{1}{p_2}z^2-\left(\frac{1}{p_2}-1\right)=z^2$.
From theorem \ref{teorema442} we have that $\mathcal{F}_h$ is not connected. Therefore, from remark \ref{obs4422} it follows that $E$ is a non-connected set.
\end{proof}

\begin{Conjecture}
There exists $0<\delta <1$ such that if $p_i>\delta$, for all $i\geq
1$ then $E$ is connected. \label{bams}
\end{Conjecture}

The conjecture \ref{bams} is motivated by the fact that the authors in \cite{ebms} proved
that there exits $0<a<1$ such that if $p_i=p>a$, for all $i\geq 1$ then $E$ is quasi-disk.

We can see some possibilities for the set $E$ in the following figures:

\begin{figure}[h!]
\centering
{\tiny 01)} \includegraphics[scale=.30]{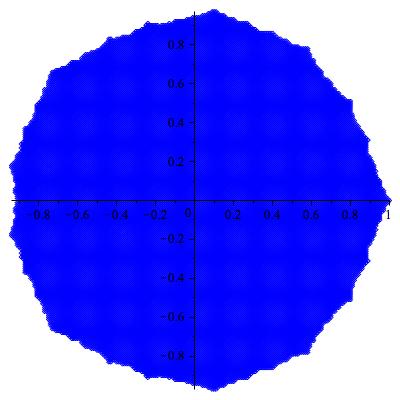}
{\tiny 02)} \includegraphics[scale=.30]{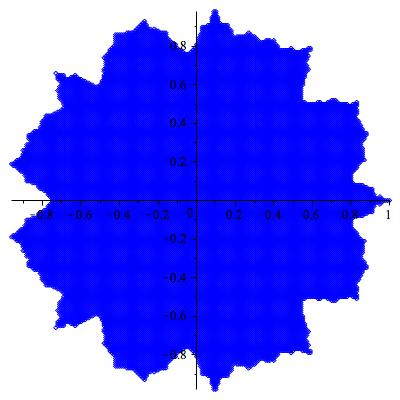}
{\tiny 03)} \includegraphics[scale=.30]{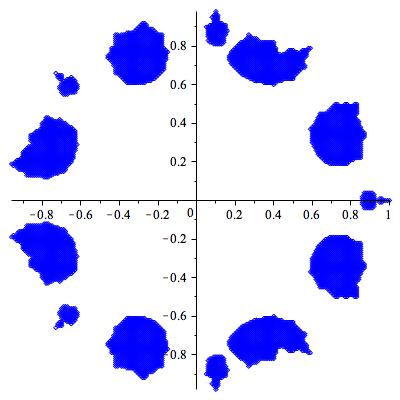}
\caption{\tiny $01)$ $p_1=1$, $p_2=0,999$, $p_3=0,909$, $p_4=0,833$, $p_5=0,769$, $p_6=0,714$, $p_7=0,666$, $p_8=0,625$, $p_9=0,588$, $p_{10}=0,555$, $p_{11}=0,526$, $p_i=1$, for all $i\in\{12,\ldots,17\}$. $02)$ $p_1=1$, $p_2=0,999$, $p_3=1$, $p_4=0,625$, $p_5=0,588$, $p_6=0,625$, $p_7=0,666$, $p_8=0,714$, $p_9=0,769$, $p_{10}=0,833$, $p_{11}=0,909$, $p_i=1$, for all $i\in\{12,\ldots,17\}$. $03)$ $p_1=1$, $p_2=0,999$, $p_3=1$, $p_4=0,588$, $p_5=0,588$, $p_6=0,625$, $p_7=0,666$, $p_8=0,714$, $p_9=0,769$, $p_{10}=0,833$, $p_{11}=0,909$, $p_i=1$, for all $i\in\{12,\ldots,17\}$.}
\end{figure}

\begin{figure}
\centering
{\tiny 04)} \includegraphics[scale=.30]{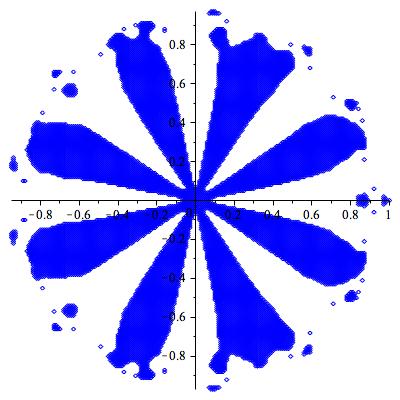}
{\tiny 05)} \includegraphics[scale=.30]{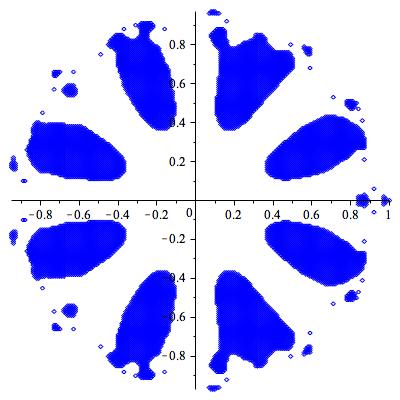}
{\tiny 06)} \includegraphics[scale=.30]{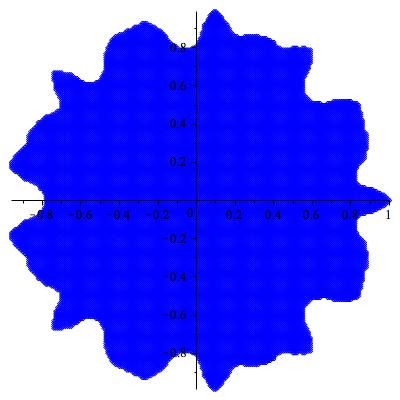}
\caption{\tiny  $04)$ $p_1=1$, $p_2=0,999$, $p_3=1$, $p_4=0,625$, $p_5=0,714$, $p_6=0,403637066$, $p_7=0,833$, $p_8=0,833$, $p_9=909$, $p_i=1$, for all $i\in\{10,\ldots,17\}$. $05)$ $p_1=1$, $p_2=0,999$, $p_3=1$, $p_4=0,625$, $p_5=0,714$, $p_6=0,403$, $p_7=0,833$, $p_8=0,833$, $p_9=909$, $p_i=1$, for all $i\in\{10,\ldots,17\}$. $06)$ $p_1=1$, $p_2=0,999$, $p_3=1$, $p_4=0,588$, $p_5=0,769$, $p_6=0,833$, $p_7=0,909$, $p_8=0,833$, $p_i=1$, for all $i\in\{9,\ldots,17\}$.}
\end{figure}

\begin{figure}
\centering
{\tiny 07)} \includegraphics[scale=.30]{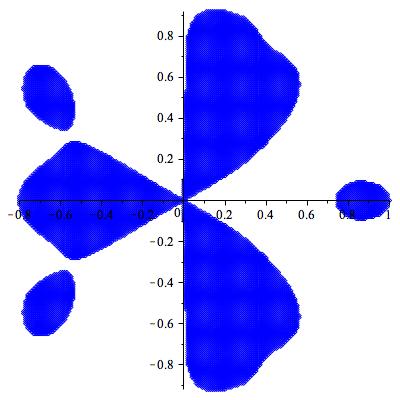}
{\tiny 08)} \includegraphics[scale=.30]{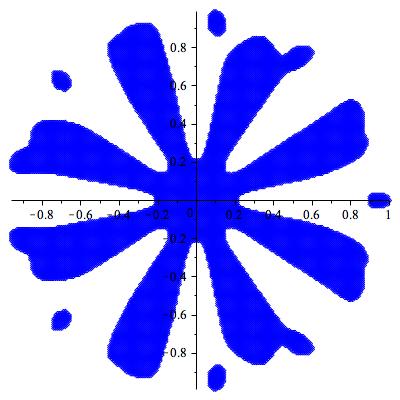}
{\tiny 09)} \includegraphics[scale=.30]{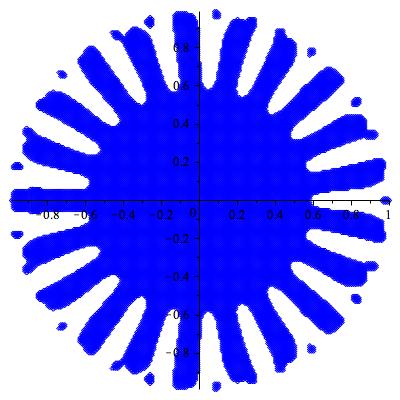}
\caption{\tiny $07)$ $p_1=1$, $p_2=0,999$, $p_3=0,5$, $p_i=1$, for all $i\in\{4,\ldots,17\}$. $08)$ $p_1=1$, $p_2=0,999$, $p_3=1$, $p_4=0,5$, $p_i=1$, for all $i\in\{5,\ldots,17\}$. $09)$ $p_1=1$, $p_2=0,999$, $p_3=1$, $p_4=1$, $p_5=0,5$, $p_i=1$, for all $i\in\{6,\ldots,17\}$.}
\end{figure}

\begin{figure}
\centering
{\tiny 10)} \includegraphics[scale=.30]{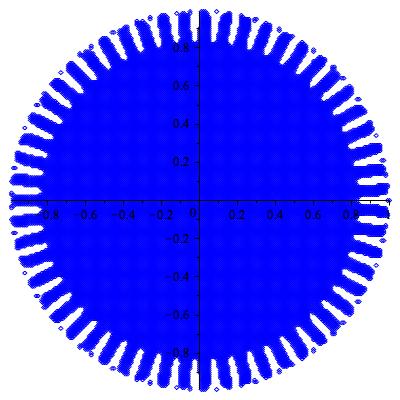}
{\tiny 11)} \includegraphics[scale=.30]{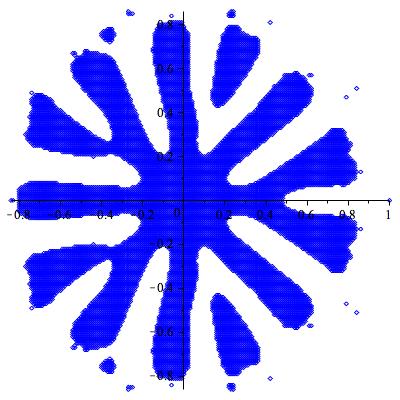}
{\tiny 12)} \includegraphics[scale=.30]{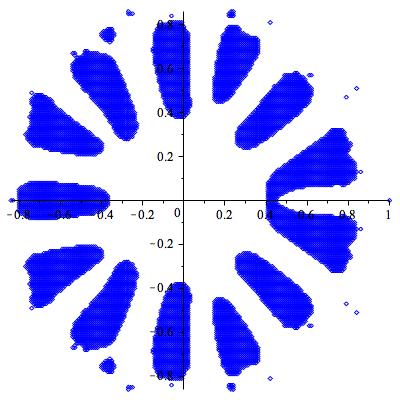}
\caption{\tiny $10)$ $p_1=1$, $p_2=0,999$, $p_3=1$, $p_4=1$, $p_5=1$, $p_6=0,5$ $p_i=1$, for all $i\in\{7,\ldots,17\}$. $11)$ $p_1=1$, $p_2=0,999$, $p_3=1$, $p_4=0,625$, $p_5=0,284864$, $p_6=0,625$, $p_7=0,714$, $p_8=0,769$, $p_9=0,833$, $p_{10}=0,909$, $p_i=1$, for all $i\in\{11,\ldots, 17\}$. $12)$ $p_1=1$, $p_2=0,999$, $p_3=1$, $p_4=0,625$, $p_5=0,284859$, $p_6=0,625$, $p_7=0,714$, $p_8=0,769$, $p_9=0,833$, $p_{10}=0,909$, $p_i=1$, for all $i\in\{11,\ldots, 17\}$.}
\end{figure}

\newpage

\begin{figure}[h!]
\centering
{\tiny 13)} \includegraphics[scale=.30]{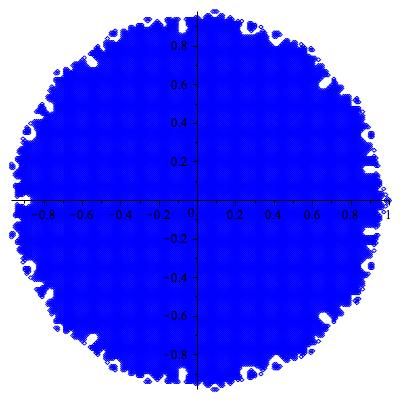}
{\tiny 14)} \includegraphics[scale=.30]{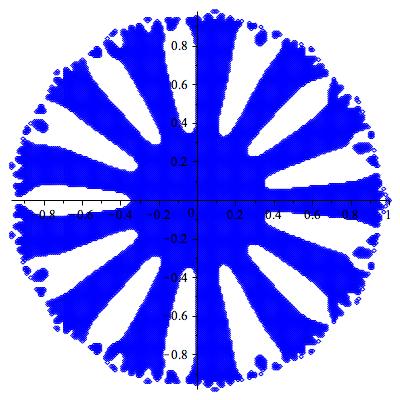}
{\tiny 15)} \includegraphics[scale=.30]{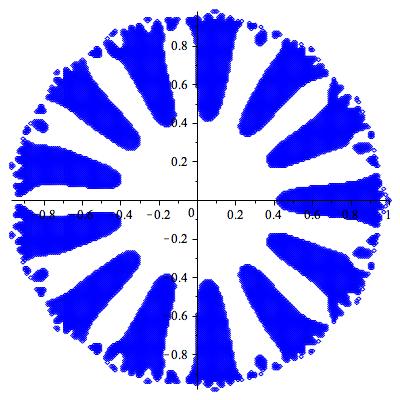}
\caption{\tiny $13)$ $p_1=1$, $p_2=0,999$, $p_3=1$, $p_4=0,833$, $p_5=0,769$, $p_6=0,833$, $p_7=0,465$, $p_8=0,833$, $p_9=909$, $p_i=1$, for all $i\in\{10,\ldots,17\}$. $14)$ $p_1=1$, $p_2=0,999$, $p_3=1$, $p_4=0,833$, $p_5=0,769$, $p_6=0,833$, $p_7=0,46041639$, $p_8=0,833$, $p_9=909$, $p_i=1$, for all $i\in\{10,\ldots,17\}$. $15)$ $p_1=1$, $p_2=0,999$, $p_3=1$, $p_4=0,833$, $p_5=0,769$, $p_6=0,833$, $p_7=0,46041617$, $p_8=0,833$, $p_9=909$, $p_i=1$, for all $i\in\{10,\ldots,17\}$.}
\end{figure}

\subsubsection*{Acknowledgments}
I heartily thank \textit{Ali Messaoudi} for the introduction to the
subject and for the many helpful comments, discussions and suggestions. I want also to thank the financial supports from \textit{FAPESP} grant $2010/19731-7$ and from \textit{CNPq} grant $150311/2015-0$.

\end{document}